\documentclass[noinfoline,ejs]{imsart}
\usepackage{amsmath,amssymb,amsthm,amsfonts,amsbsy}
\usepackage{graphicx}

\usepackage{enumerate}
\usepackage[final]
{showkeys}

\usepackage{hyperref}

\providecommand{\texorpdfstring}[2]{#1}
\providecommand{\url}[1]{#1}

\renewcommand{\le}{\leqslant}
\renewcommand{\ge}{\geqslant}

\newcommand{\ffrown}{\text{\raisebox{3pt}[0pt][0pt]{$\frown$}}}
\renewcommand{\O}{\underset{\ffrown}{<}}
\newcommand{\OG}{\underset{\ffrown}{>}}

\renewcommand{\P}{\operatorname{\mathsf{P}}} 
\newcommand{\E}{\operatorname{\mathsf{E}}}

\newcommand{\cc}{\mathsf{c}}    

\newcommand{\Ko}{\mathsf{Ko}}
\newcommand{\bW}{\mathsf{bW}}

\newcommand{\R}{\mathbb{R}}                  
\newcommand{\N}{\mathbb{N}}

\newcommand{\al}{\alpha}
\newcommand{\be}{\beta}
\newcommand{\ga}{\gamma}
\newcommand{\Ga}{\Gamma}
\newcommand{\de}{\delta}

\newcommand{\ep}{\epsilon}
\newcommand{\vp}{\varepsilon}

\renewcommand{\th}{\theta}
\newcommand{\Th}{\Theta}

\newcommand{\la}{\lambda}

\newcommand{\si}{\sigma}

\newcommand{\vsi}{\varsigma}
\newcommand{\om}{\omega}
\newcommand{\Om}{\Omega}

\renewcommand{\j}{{\mathbf{j}}}
\newcommand{\x}{{\mathbf{x}}}
\newcommand{\y}{{\mathbf{y}}}
\newcommand{\Y}{{\mathbf{Y}}}
\newcommand{\ZZ}{{\mathbf{Z}}}

\newcommand{\tsi}{{\tilde\sigma}}

\newcommand{\CC}{\mathfrak{C}}
\newcommand{\XX}
{\R^d}

\newcommand{\A}{\mathcal{A}}
\newcommand{\B}{\mathcal{B}}
\newcommand{\D}{\mathcal{D}}

\newcommand{\XXX}{\mathcal{X}}

\newcommand{\BB}{\mathrm{B}}

\newcommand{\xx}{\mathbf{x}}
\newcommand{\X}{\mathbf{X}}

\newcommand{\0}{\mathbf{0}}

\newcommand{\dd}{\operatorname{d}\!}
\newcommand{\sign}{\operatorname{sign}}
\newcommand{\ii}[1]{\operatorname{I}\{#1\}}

\newcommand{\tto}[1]{\underset{#1}\longrightarrow}

\newcommand{\Mf}{M_\ep}

\newcommand{\s}[1]{{^{(#1)}}}
\renewcommand{\ss}[2]{{_{#1}^{(#2)}}}

\renewcommand{\bar}{\overline}

\newtheorem{theorem}{Theorem}

\newtheorem{lemma}[theorem]{Lemma}
\newtheorem{proposition}[theorem]{Proposition}
\theoremstyle{definition}
\newtheorem{remark}[theorem]{Remark}
\newtheorem{example}[theorem]{Example}

\numberwithin{equation}{section}
\numberwithin{theorem}{section}

\begin{document} 

\begin{frontmatter}

\title{Optimal-order uniform and nonuniform bounds on the rate of convergence to normality for maximum likelihood estimators\protect\thanksref{T1}}
\runtitle{Rate of convergence to normality for maximum likelihood estimators}
\thankstext{T1}{\today}

\begin{aug}
 \author{\fnms{Iosif}  \snm{Pinelis}\corref{}\ead[label=e1]{ipinelis@mtu.edu}} 
 \runauthor{I. Pinelis}
 \affiliation{Michigan Technological University}
 \address{Department of Mathematical Sciences\\Michigan Technological University\\Hough\-ton, Michigan 49931\\\printead{e1}}
\end{aug} 

\begin{abstract} 
It is well known that, under general regularity conditions, the distribution of the maximum likelihood estimator (MLE) is asymptotically normal. 
Very recently, bounds of the optimal order $O(1/\sqrt n)$ on the closeness of the distribution of the MLE to normality in the so-called bounded Wasserstein distance were obtained 
\cite{anast-rein_publ,anast-ley}, where $n$ is the sample size. However, the corresponding bounds on the Kolmogorov distance were only of the order $O(1/n^{1/4})$. In this 
paper, bounds of the optimal order $O(1/\sqrt n)$ on the closeness of the distribution of the MLE to normality in the Kolmogorov distance are given, as well as their nonuniform counterparts, which work better 
in tail zones of the distribution of the MLE. These results are based in part on previously obtained general optimal-order bounds on the rate of convergence to normality in the multivariate delta method. The crucial observation is that, under natural conditions, the MLE can be tightly enough bracketed between two smooth enough functions of the sum of independent random vectors, which makes the delta method applicable. It appears that the nonuniform bounds for MLEs in general have no precedents in the existing literature; a special case was recently treated by Pinelis and Molzon 
\cite{nonlinear-publ}. The results can be extended to $M$-estimators. 
\end{abstract}

\setattribute{keyword}{AMS}{AMS 2010 subject classifications:}

\begin{keyword}[class=AMS]
\kwd{62F10}
\kwd{62F12}
\kwd
{60F05}
\kwd{60E15}
\end{keyword}




\begin{keyword}
\kwd{maximum likelihood estimators}
\kwd{Berry--Esseen bounds}
\kwd{delta method}
\kwd{rates of convergence}
\end{keyword}

\end{frontmatter}

\tableofcontents 

%

\section{Introduction}
\label{intro}

Let us begin with the following quote from Kiefer \cite{kiefer68} of 1968: 

\begin{quote} 
a second area of what seem to me important problems to work on
has to do with the fact that we do have, in many settings, quite a good large sample
theory, but we don't know how large the sample sizes have to be for
that theory to take hold. Now, I'm sure most of you are familiar with the
error estimate one can give for the classical central-limit theorem, which goes
by the name of the Berry-Esseen estimate, and which tells you that under
certain assumptions one can actually give an explicit bound on the departure
from the normal distribution of the sample mean for a given sample size,
the error term being of order $1/\sqrt n$. For most other statistical problems, in
fact for almost anything other than the use of the sample mean, we have
nothing. The most obvious example of this (and this is not original with me;
many people have been concerned with this), is the maximum likelihood
estimator in the case of regular estimation. We all know what the asymptotic
distribution is. Can you give explicitly some useful bound on the departure
from the asymptotic normal distribution as a function of the sample size $n$?
It seems to be a terrifically difficult problem.
\end{quote} 

Since then, there has been some significant progress in this direction, especially rather recently. For instance, Berry--Esseen-type bounds of order $1/\sqrt n$ were obtained for $U$-statistics -- see e.g.\ \cite{kor94}; for the Student statistic \cite{bent96,bbg96}; and, even more recently, for rather broad classes of other statistics that depend on the observations in a nonlinear fashion \cite{chen07,nonlinear-publ}. 

As Kiefer pointed out, it is well known that, under general regularity conditions, the distribution of the maximum likelihood estimator (MLE) is asymptotically normal. 
In this paper, we shall consider Berry--Esseen-type bounds of order $1/\sqrt n$ for the MLE. First such bounds were apparently obtained in the paper 
\cite{michel-pfanzagl71}, followed by \cite{pfanzagl71,pfanzagl73}.   
%
Very recently, 
bounds on the closeness of the distribution of the MLE to normality in the so-called bounded Wasserstein distance, $d_{\bW}$, were obtained in \cite{anast-rein_publ}. 
In the rather common special case when the MLE $\hat\th$ is expressible as a smooth enough function of a linear statistic of 
independent identically distributed (i.i.d.) observations, the bounds obtained in \cite{anast-rein_publ} were sharpened and simplified in \cite{anast-ley} by using a version of the delta method. 
More specifically, it was assumed in \cite{anast-ley} that 
\begin{equation}\label{eq:q(th)}
 q(\hat\th)=\frac1n\,\sum_{i=1}^n g(X_i),  
\end{equation}
where $q\colon\Th\to\R$ is a twice continuously differentiable one-to-one mapping, $g\colon\R\to\R$ is a Borel-measurable function, and the $X_i$'s 
are i.i.d.\  
real-valued r.v.'s.  

It was noted in \cite[Proposition~2.1]{anast-rein_publ} that for any r.v.\ $Y$ 
and a standard normal r.v.\ $Z$ one has $d_\Ko(Y,Z)\le2\sqrt{d_{\bW}(Y,Z)}$, where $d_\Ko$ denotes the Kolmogorov distance. 
This bound on $d_\Ko$ in terms of $d_{\bW}$ is the best possible one, up a constant factor, as shown in \cite{nonlinear-publ}. 
Therefore, even though the bounds on the bounded Wasserstein distance $d_{\bW}$ obtained in \cite{anast-rein_publ,anast-ley} are of the optimal order $O(1/\sqrt n)$, the resulting bounds on the Kolmogorov distance are only of the order $O(1/n^{1/4})$. 
\big(That the order $O(1/\sqrt n)$ is optimal for MLEs is well known; for instance, see the example of the Bernoulli family of distributions given in \cite{michel-pfanzagl71}.\big)

In \cite{nonlinear-publ}, optimal-order bounds of the form $O(1/\sqrt n)$ on the rate of convergence to
normality in the 
general multivariate delta method were given. Those results are applicable when the statistic of interest can be expressed as a smooth enough function of the sum of independent random vectors. 
Accordingly, various kinds of applications were presented in \cite{nonlinear-publ}. 
In particular,  
uniform and nonuniform  
bounds of the optimal order 
on the closeness of the distribution of the MLE to normality 
were obtained in \cite{nonlinear-publ} under conditions similar to the mentioned conditions assumed in \cite{anast-ley}. 

In this paper we present a way to extend those results in \cite{nonlinear-publ} to the general case, without an assumption of the form \eqref{eq:q(th)}, made in \cite{anast-ley,nonlinear-publ}. Of course, in general the MLE cannot be represented as a function of the sum of independent random vectors (see Appendix~\ref{append} for details). However, the crucial observation here is that, under natural conditions, the MLE can be tightly enough bracketed between two such smooth enough functions, which makes the delta method applicable. 
Thus, the present paper is methodologically 
different from the preceding work on Berry--Esseen-type bounds for the MLE, in that it relies on the general result developed in \cite{nonlinear-publ}, rather than on methods specially designed to deal with the MLE. 

Perhaps more importantly, the new method yields 
not only uniform 
bounds (that is, in the Kolmogorov metric) of the optimal order $O(1/\sqrt n)$ 
on the closeness of the distribution of the MLE to normality 
but also their so-called nonuniform counterparts, which work much 
better for large deviations, that is, in tail zones of the distribution of the MLE -- which are usually of foremost interest in statistical tests. 
Such nonuniform bounds for MLEs in general appear to have no precedents in the existing literature (except that, as stated above, a special case of nonuniform bounds for MLEs was recently treated in \cite{nonlinear-publ}). 

The paper is organized as follows. The general setting of the problem is described in Section~\ref{setting}. The key step of tight enough bracketing of the MLE between two functions of the sum of independent random vectors is made in Section~\ref{bracketing}. General uniform and nonuniform optimal-order bounds from \cite{nonlinear-publ} on the 
convergence rate in the multivariate delta method are presented in Section~\ref{f(bar V)}. In Section~\ref{appl}, we make the bracketing work by applying the general bounds in the multivariate delta method. Yet, this leaves out the problem of bounding 
a remainder, which is a probability of large deviations of the MLE from the true value of the parameter. It is shown in Section~\ref{remainder} that under natural conditions this remainder is
exponentially fast decreasing (in $n$) and thus asymptotically negligible as compared to the main term on the order of $1/\sqrt n$. 
All these findings are summarized in Section~\ref{concl}, where the main result of this paper is presented, along with corresponding discussion.  
In Appendix~\ref{append}, it is shown that, under general regularity conditions, \eqref{eq:q(th)} (or even a relaxed version of it) implies that the family of densities is a one-parameter exponential one;  
in particular, this allows one to give any number of examples where the main result of the present paper is applicable, whereas the corresponding result in \cite{nonlinear-publ} is not.

\section{General setting}
\label{setting}

Let 
$X,X_1,X_2,\dots$ be random variables (r.v.'s) mapping a measurable space $(\Om,\A)$ to another measurable space $(\XXX,\B)$ and let $(\P_\th)_{\th\in\Th}$ be a parametric family 
of probability measures on $(\Om,\A)$ such that the r.v.'s $X,X_1,X_2,\dots$ are i.i.d.\ with respect to each of the probability measures $\P_\th$ with $\th\in\Th$; here the parameter space $\Th$ is assumed to be a subset of the real line $\R$. 
As usual, let $\E_\th$ denote the expectation with respect to the probability measure $\P_\th$. 
Suppose that for each $\th\in\Th$ the distribution $\P_\th X^{-1}$ of $X$ has a density $p_\th$ with respect to a measure $\mu$ on $\B$. 
Because the extended real line $[-\infty,\infty]$ is compact, for each $n\in\N$ and each point $\xx=\xx_n=(x_1,\dots,x_n)\in\XXX^n$ the likelihood function $\Th\ni\th\mapsto L_\xx(\th):=\prod_{i=1}^n p_\th(x_i)$ has at least one 
generalized maximizer $\hat\th_n(\xx)$ in the closure of the set $\Th$ in $[-\infty,\infty]$, in the sense that $\sup_{\th\in\Th}L_\xx(\th)=\limsup_{\th\to\hat\th_n(\xx)}L_\xx(\th)$. Picking, for each $\xx=(x_1,\dots,x_n)\in\XXX^n$, any one of such generalized maximizers $\hat\th_n(\xx)$, one obtains a map $\Om\ni\om\mapsto\hat\th_n(\X(\om))$, where $\X:=\X_n:=(X_1,\dots,X_n)$; any such map will be denoted here by $\hat\th_n(\X)$ (or simply by 
$\hat\th_n$ or $\hat\th$) and referred to as a maximum likelihood estimator (MLE) of $\th$.   
This is a somewhat more general definition of the MLE than usual, and in general an MLE $\hat\th$ will not have to be a r.v.; that is, it can be non-measurable with respect to the sigma-algebra $\A$. 
However, to simplify the presentation, we shall still refer to sets of the form $\{\hat\th\in J\}:=\{\om\in\Om\colon\hat\th_n(\X(\om))\in J\}$ for Borel sets $J\subseteq\Th$ as events and write $\P_\th(\hat\th\in J)$ implying that the latter expression may and should be understood as either one of the expressions $(\P_\th)^*(\hat\th\in J)$ or $(\P_\th)_*(\hat\th\in J)$, where ${}^*$ and ${}_*$ stand for the corresponding outer and inner measures. Of course, when the map $\hat\th$ is measurable, then one can use the bona fide expressions of the mentioned form $\P_\th(\hat\th\in J)$.

Let $\th_0\in\Th$ be the ``true'' value of the unknown parameter $\th$, such that  
\begin{equation}\label{eq:in Th}
[\th_0-\de,\th_0+\de]\subseteq\Th^\circ	
\end{equation}
for some real $\de>0$, where $\Th^\circ$ denotes the interior of the subset $\Th$ of $\R$. 
For brevity, let 
\begin{equation*}
	\P:=\P_{\th_0}\quad\text{and}\quad\E:=\E_{\th_0}. 
\end{equation*}

For $x\in\XXX$ and $\th\in\Th$, consider the log-likelihood  
\begin{equation*}
	\ell_x(\th):=\ln p_\th(x)
\end{equation*}
and assume the following: 
\begin{enumerate}[(I)]
\item \label{diff} The set $\XXX_{>0}:=\{x\in\XXX\colon p_\th(x)>0\}$ is the same for all $\th\in[\th_0-\de,\th_0+\de]$, and for each $x\in\XXX_{>0}$ the density $p_\th(x)$ and hence the log-likelihood $\ell_x(\th)$ are thrice differentiable in $\th$ at each point $\th\in[\th_0-\de,\th_0+\de]$. 
	\item \label{fisher} Standard regularity conditions hold so that $\E \ell'_X(\th_0)=0$ and $\E \ell'_X(\th_0)^2=-\E \ell''_X(\th_0)
	=I(\th_0)\in(0,\infty)$, where $I(\th)$ is the Fisher information at $\th$. 
	\item \label{M_2} $
	\E |\ell'_X(\th_0)|^3+
	\E |\ell''_X(\th_0)|^3<\infty$. 
	\item \label{M_3} 
	$\E \sup\limits_{\th\in[\th_0-\de,\th_0+\de]}|\ell'''_X(\th)|^3<\infty$. 
\end{enumerate}

\begin{remark}\label{rem:fisher}
The expectation $\E \ell'_X(\th_0)$, mentioned in condition~\eqref{fisher}, may be understood as $\int_{\XXX_{>0}}p'_x(\th_0)\mu(\dd x)$, where $p_x(\th):=p_\th(x)$; similarly, for the other expectations mentioned in conditions~\eqref{fisher}--\eqref{M_3}. Of course, all the derivatives here are with respect to $\th$. 

Concerning the ``standard regularity conditions'' mentioned in condition~\eqref{fisher}, it will be enough to assume that $\P(\frac\partial{\partial\th}p_\th(X)\ne0)>0$ and for some measurable function $g\colon\XXX_{>0}\to[0,\infty)$ such that $\int_{\XXX_{>0}}g\dd\mu<\infty$ and all $\th\in[\th_0-\de,\th_0+\de]$ and $x\in\XXX_{>0}$ we have $|\frac\partial{\partial\th}p_\th(x)|+|\frac{\partial^2}{\partial\th^2} p_\th(x)|\le g(x)$; see e.g.\ \cite[Lemma~5.3, page~116]{lehmann-estim} and \cite[Lemma~2.4]{rosenthal_AOP} (more general conditions can be given using \cite[Lemma~2.3]{rosenthal_AOP}). 
Then $I(\th)$ will also be continuous in $\th\in[\th_0-\de,\th_0+\de]$. 

Conditions \eqref{diff}--\eqref{M_3} are rather similar to regularity conditions used in related literature; see Remark~\ref{rem:compare} on page~\pageref{rem:compare} for details. 
It appears that these conditions will be generally satisfied provided that $\ell_x(\th)$ is smooth enough in $\th$. 

For instance, let us briefly consider the case when the family of densities $(p_\th)$ is a location family, so that $\ell_x(\th)=\la(x-\th)$ for all $(x,\th)\in\XXX\times\Th=\R^2$, where $\la$ is a smooth enough function. If the densities $p_\th$ have power-like tails, then for some positive real constants $c_+$ and $c_-$ one has
$\la(x)\sim-c_{\pm}\ln|x|$ as $x\to\pm\infty$, 
in which case typically 
$|\la^{(k)}(x)|\sim 
-c_\pm k!|x|^{-k}\ln|x|$ for $k=0,1,\dots$ as 
$x\to{\pm}\infty$. 
So, conditions \eqref{M_2} and \eqref{M_3} will hold, since $|\ell_x^{(k)}(\th)|=|\la^{(k)}(x-\th)|$. 
If the tails densities $p_\th$ are lighter than power-like tails, so that (say) 
$\la(x)\sim-c_{\pm}|x|^\al$ for some real $\al>0$ as $x\to\pm\infty$, 
then typically 
$|\la^{(k)}(x)|\sim 
-c_\pm k!|x|^{\al-k}$ for $k=0,1,\dots$ as 
$x\to{\pm}\infty$, so that conditions \eqref{M_2} and \eqref{M_3} will again hold. 

The case of a scale family is quite similar to that of a location family. Alternatively, the ``scale'' case can be reduced to the ``location'' one by logarithmic rescaling in both $x$ and $\th$. 

At this point, consider also the case when the family of densities $(p_\th)$ is an exponential family,  
so that $\ell_x(\th)=w(\th)T(x)+d(\th)$ for some functions $w$, $T$, and $d$ and for all $(x,\th)\in\XXX\times\Th=\R^2$, where the functions $w$ and $d$ are smooth enough, with $w'(\th_0)\ne0$. Then $\ell_x^{(k)}(\th)=w^{(k)}(\th)T(x)+d^{(k)}(\th)$. So, conditions \eqref{M_2} and \eqref{M_3} will hold in this case as well, since $\E|T(X)|^\al=\int_\XXX|T(x)|^\al\exp\{w(\th_0)T(x)+d(\th_0)\}\mu(\dd x)$ for $\al>0$, $|T(x)|^\al=O(e^{hT(x)}+e^{-hT(x)})$ for any given real $\al>0$ and any given nonzero real $h$, 
and the conditions $\th_0\in\Th^\circ$ and $w'(\th_0)\ne0$ imply that $\int_\XXX\exp\{[w(\th_0)+h]T(x)+d(\th_0)\}\mu(\dd x)<\infty$ for all real $h$ close enough to $0$.  \qed
\end{remark}


Let 
\begin{equation}\label{eq:ell}
	\ell_\X(\th):=\sum_{i=1}^n\ell_{X_i}(\th)  
\end{equation}
for $\th\in\Th$, 
the log-likelihood of the sample $\X=(X_1,\dots,X_n)$. 

\section{Tight bracketing of the MLE between two functions of the sum of independent random vectors}
\label{bracketing}
 
Without loss of generality (w.l.o.g.), $\XXX_{>0}=\XXX$. 
Then on the event 
\begin{equation}\label{eq:G}
G:=\{\hat\th\in[\th_0-\de,\th_0+\de]\} 	
\end{equation}
($G$ for ``\underline{g}ood event'')
one must have 
\begin{align}
	0=\ell'_\X(\hat\th)
	=&\ell'_\X(\th_0)
	+(\hat\th-\th_0)\,\ell''_\X(\th_0)
		+\frac{(\hat\th-\th_0)^2}2\,\ell'''_\X(\th_0+\xi(\hat\th-\th_0)) \label{eq:0=} \\ 
	=&n\Big(\bar Z-(\hat\th-\th_0)\,\bar U
		+\frac{(\hat\th-\th_0)^2}2\,\bar R\Big)	\label{eq:=n()}
\end{align}
for some $\xi\in(0,1)$, depending on the values of the $X_i$'s, where $\bar Z:=\frac1n\sum_{i=1}^n Z_i$, $\bar U:=\frac1n\sum_{i=1}^n U_i$, $\bar R:=\frac1n\sum_{i=1}^n R_i$, $\bar{R^*}:=\frac1n\sum_{i=1}^n R_i^*$, 
\begin{equation}\label{eq:Z,R}
\begin{gathered}
	Z_i:=\ell'_{X_i}(\th_0),\quad U_i:=-\ell''_{X_i}(\th_0), \\  R_i:=\ell'''_{X_i}(\th_0+\xi(\hat\th-\th_0))\in[-R_i^*,R_i^*],\quad R_i^*:=\sup\limits_{\th\in[\th_0-\de,\th_0+\de]}|\ell'''_{X_i}(\th)|. 
\end{gathered}	
\end{equation}
Note that the $Z_i$'s are i.i.d.\ r.v.'s, and so are the $U_i$'s and the $R_i^*$'s (but not necessarily the $R_i$'s). 

Equalities \eqref{eq:0=} and \eqref{eq:=n()} provide a quadratic equation for $\hat\th$. So, on the event $G$ one has 
\begin{equation}\label{eq:cases}
	\begin{alignedat}{2}
	\hat\th-\th_0&=\frac{\bar Z}{\bar U} &&\ \text{ if }\ \bar R=0\ \&\ \bar U\ne0, \\
	\hat\th-\th_0&\in\{d_+,d_-\} &&\ \text{ if }\ \bar R\ne0, 
	\end{alignedat}
\end{equation}
where 
\begin{equation*}
	d_\pm:=\frac{\bar U\pm\sqrt{\bar U^2-2\bar Z\,\bar R}}{\bar R}. 
\end{equation*}

Letting 
\begin{multline}\label{eq:B:=}
	B:=B_1\cup B_2,\quad\text{where} \\  B_1:=\{\bar R\ne0,\;\hat\th-\th_0=d_+\}\cup\{\bar U\le0\} 
	\quad\text{and}\quad B_2:=\{
	\bar U^2\le 2|\bar Z|\,\bar{R^*}\}   
\end{multline}
($B$ for ``\underline{b}ad event''), 
on the event $B_1\cap\{\bar U>0\}$ 
one has $|\hat\th-\th_0|=|d_+|\ge\bar U/|\bar R|\ge\bar U/\bar{R^*}$, whence, by \eqref{eq:G}, 
\begin{equation}\label{eq:G cap B}
	\P (G\cap B_1)\le\P \Big(\bar U\le0\text{ or }\frac{\bar U}{\bar{R^*}}\le\de\Big)
	=\P\Big(\frac{\bar U}{\bar{R^*}}\le\de\Big)
	=\P\Big(\sum_{i=1}^n(U_i-\de R_i^*)\le0\Big). 
\end{equation}
By definitions \eqref{eq:Z,R} and conditions \eqref{fisher}, \eqref{M_2}, and \eqref{M_3},  
\begin{equation}\label{eq:^{3/2}<infty}
\E U_1>0,\quad\E |Z_1|^3<\infty,\quad \E |U_1|^3<\infty,\quad \E (R_1^*)^3<\infty, 	
\end{equation}
and hence $\E R_1^*<\infty$. 
So, w.l.o.g.\ one may choose $\de>0$ to be small enough so that 
\begin{equation*}
	\de_1:=\E (U_i-\de R_i^*)>0. 
\end{equation*}
Then, letting $Y_i:=(U_i-\de R_i^*)-\E(U_i-\de R_i^*)$ and using \eqref{eq:G cap B}, 
Markov's inequality, and a Rosenthal-type inequality (see e.g.\ \cite[Theorem~1.5]{rosenthal_AOP})
\begin{multline}\label{eq:P(G and B_1)}
	\P (G\cap B_1)\le\P \Big(\sum_{i=1}^n Y_i\le-n\de_1\Big)
	\le\frac1{(n\de_1)^3}\,\E \Big|\sum_{i=1}^n Y_i\Big|^3 \\ 
	\le\frac{n\E |Y_1|^3+\sqrt{8/\pi}\,(n\E Y_1^2)^{3/2}}{(n\de_1)^3}\le\frac\CC{n^{3/2}},  
\end{multline}
where $\CC:=\big(\E |Y_1|^3+\sqrt{8/\pi}\,(\E Y_1^2)^{3/2}\big)/\de_1^3$, which depends on $\de_1>0$, $\E Y_1^2<\infty$, and $\E |Y_1|^3<\infty$ -- but not on $n$. 

Next, the occurrence of $B_2$ implies that 
of at least one of the following events: $B_{21}:=\{\bar U\le\frac12\,\E U_1\}$, $B_{22}:=\{\bar{R^*}\ge1+\E R_1^*\}$, or 
$B_{23}:=\{|\bar Z|\ge\frac18\,(\E U_1)^2/(1+\E R_1^*)\}$. 
So,
\begin{equation}\label{eq:B_2j}
	\P(B_2)\le\P(B_{21})+\P(B_{22})+\P(B_{23}). 
\end{equation} 
In view of \eqref{eq:^{3/2}<infty}, 
the bounding of each of the probabilities $\P(B_{21})$, $\P(B_{22})$, $\P(B_{23})$ is quite similar to 
the bounding of $\P (G\cap B_1)$ in \eqref{eq:P(G and B_1)}
 -- because \break
$\P(B_{21})=\P(\sum_{i=1}^n Y_{i,21}\le-n\de_{21})$, $\P(B_{22})=\P(\sum_{i=1}^n Y_{i,22}\ge n\de_{22})$, and  $\P(B_{23})=\P(\sum_{i=1}^n |Y_{i,23}|\ge n\de_{23})$, where 
$Y_{i,21}:=U_i-\E U_1$, $\de_{21}:=\frac12\,\E U_1>0$, 
$Y_{i,22}:=R_i^*-\E R_1^*$, $\de_{22}:=1>0$, 
$Y_{i,23}:=Z_i-\E Z_1=Z_i$, $\de_{23}:=\frac18\,(\E U_1)^2/(1+\E R_1^*)>0$. 

Thus, by \eqref{eq:B:=}, \eqref{eq:P(G and B_1)}, and \eqref{eq:B_2j}, 
\begin{equation}\label{eq:P(G and B)}
	\P (G\cap B)\le\P(G\cap B_1)+\P(B_2)\le\frac\CC{n^{3/2}} . 
\end{equation}

On the other hand, if $\bar R\ne0$ and $\bar U>0$, then $d_-=\frac{2\bar Z}{\bar U+\sqrt{\bar U^2-2\bar Z\,\bar R}}$; here, the condition $\bar U>0$ was used only to ensure that the denominator of the latter ratio is nonzero. Hence, on the event $G\setminus B$ one has  
\begin{equation}\label{eq:hath-th0=}
	\bar U>0 \quad\text{and}\quad 
	\hat\th-\th_0=\frac{2\bar Z}{\bar U+\sqrt{\bar U^2-2\bar Z\,\bar R}}\in[T_-,T_+],   
\end{equation}
where 
\begin{equation}\label{eq:T_pm}
	T_\pm:=\frac{2\bar Z}{\bar U+\sqrt{\bar U^2\mp2|\bar Z|\,\bar{R^*}}};  
\end{equation}
note that, when $\bar R=0$ and $\bar U>0$, the expression of $\hat\th-\th_0$ in \eqref{eq:hath-th0=} is in agreement with the corresponding expression in \eqref{eq:cases}.  

Now that the desired bracketing 
of $\hat\th-\th_0$ between $T_-$ and $T_+$ is obtained in \eqref{eq:hath-th0=}, we are ready to apply some of the mentioned general results of \cite{nonlinear-publ}, presented in the next section. 

\section{General uniform and nonuniform bounds from \texorpdfstring{\cite{nonlinear-publ}}{} on the rate of convergence to normality for smooth nonlinear functions of sums of independent random vectors}
\label{f(bar V)}

The standard normal distribution function (d.f.) will be denoted by $\Phi$. 
For any $\R^d$-valued random vector $\zeta$, we use the norm notation 
\begin{equation*}
 \|\zeta\|_p:=\big(\E\|\zeta\|^p\big)^{1/p}\text{ for any real $p\ge1$}, 
\end{equation*}
where $\|\cdot\|$ denotes the Euclidean norm on $\R^d$. 

Take any Borel-measurable functional $f\colon\XX\to\R$ satisfying the following smoothness condition: there exist $\ep\in(0,\infty)$, $\Mf\in(0,\infty)$, and a continuous linear functional $L\colon\XX\to\R$ such that
\begin{align}\label{eq:smooth}
 |f(\x)-L(\x)|\le\frac \Mf2\,\|\x\|^2\text{ for all $\x\in\XX$ with }\|\x\|\le\ep.
\end{align}
Thus, 
$f(\0)=0$ and $L$ necessarily coincides with the first Fr\'echet derivative, $f'(\0)$, of the function $f$ at $\0$. 
Moreover, for the smoothness condition \eqref{eq:smooth} to hold, it is enough that 
\begin{equation}\label{eq:M^*}
	\Mf\ge\Mf^*:=\sup\bigg\{\frac1{\|\x\|^2}\,\bigg|\frac{\dd^2}{\dd t^2}\,f(\x+t\x)\Big|_{t=0}\bigg|\colon \x\in\XX,\,0<\|\x\|\le\ep\bigg\}
	;  
\end{equation}
it is not necessary that $f$ be twice differentiable at $\0$. E.g., if $d=1$ and $f(x)=\frac{x}{1+|x|}$ for $x\in\R$, then $f(0)=0$, $f'(0)=1$, and $f''(x)=-\frac{2\sign x}{(1+|x|)^3}$ for real $x\ne0$; so, \eqref{eq:smooth} holds for any real $\ep>0$ with $L(x)\equiv x$ and $\Mf=2$, whereas $f''(0)$ does not exist.  

\bigskip

Let
$
 V,V_1,\dotsc,V_n\text{ be i.i.d.\ random vectors}
$ 
in $\XX$, with $\E V=\0$ and  
\begin{equation*}
 \bar V:=\frac{1}{n}\sum_{i=1}^nV_i. 
\end{equation*} 
Further let
\begin{equation}\label{eq:tsi,v_p,vsi_p}
 \tsi:=\|L(V)\|_2,\quad v_3:=\|V\|_3,\quad\text{and}\quad\vsi_3:=\frac{\|L(V)\|_3}{\tsi}. 
\end{equation}

\begin{theorem}\label{th:nonlin}
\ \kern-9pt\emph{\cite{nonlinear-publ}}
Suppose that \eqref{eq:smooth} holds, and that $\tsi>0$ and $v_3<\infty$. 
Then for all $z\in\R$
\begin{equation}\label{eq:f(S).iid}
 \Big|\P\Big(\frac{f(\bar V)}{\tsi /\sqrt n}\le z\Big)-\Phi(z)\Big|\le\frac{\CC}{\sqrt{n}}, 
\end{equation}
where $\CC$ is a finite positive expression that depends only on the function $f$ (through \eqref{eq:smooth}) and the moments $\tsi$, $\vsi_3$, and $v_3$. 
Moreover, for any $\om\in(0,\infty)$ and for all 
\begin{equation}\label{eq:z,iid}
z\in\bigl(0,\om\,\sqrt{n}\,\bigr]
\end{equation}
one has
\begin{align}
\label{eq:f(S).iid.power} 
 \Big|\P\Big(\frac{f(\bar V)}{\tsi /\sqrt n}\le z\Big)-\Phi(z)\Big|  &\le\frac{\CC_\om}{z^3\,\sqrt{n}}, 
\end{align}
where $\CC_\om$ is a finite positive expression that depends only on the function $f$ (through \eqref{eq:smooth}), the moments $\tsi$, $\vsi_3$, and $v_3$, and also on $\om$. 
\end{theorem}

The restriction \eqref{eq:z,iid} cannot be relaxed in general; see 
\cite{nonlinear-publ}. 


To simplify the presentation, in what follows let $\CC$ stand for various finite positive expressions whose values do not depend on $n$ or $z$; that is, $\CC$ will denote various positive real constants -- with respect to $n$ and $z$. However, $\CC$ may depend on other attributes of the setting, including the model $(\P_\th)_{\th\in\Th}$ under consideration, the $\P_{\th_0}$-distribution of $X_1$, and the values of parameters freely chosen in a given range (such as $\om$ in \eqref{eq:z,iid} and $\vp$ in \eqref{eq:smooth}).   

\section{Making the bracketing work: Applying the general bounds of \texorpdfstring{\cite{nonlinear-publ}}{} }
\label{appl}

Now let $d=3$ and then let 
\begin{equation*}
	\D:=\{\x=(x_1,x_2,x_3)\in\XX=\R^3\colon x_2+\E  U_1>0,\ (x_2+\E U_1)^2>2|x_1|\,|x_3+\E R_1^*|\}. 
\end{equation*}
By \eqref{eq:Z,R} and conditions \eqref{fisher} and \eqref{M_3}, $\E U_1=I(\th_0)\in(0,\infty)$ and $\E R_1^*\in[0,\infty)$. 
So, for some real $\ep>0$, the set $\D$ contains the $\ep$-neighborhood of the origin $\0$ of $\R^3$. 

Define functions $f_\pm\colon\R^3\to\R$ by the formula 
\begin{equation}\label{eq:f_pm}
	 f_\pm(\x)=f_\pm(x_1,x_2,x_3)=\frac{2x_1}{x_2+\E  U_1+\sqrt{(x_2+\E U_1)^2\mp2|x_1|\,|x_3+\E R_1^*|}} 
\end{equation}
for $\x=(x_1,x_2,x_3)\in\D$, and let $f(\x):=0$ if $\x\in\R^3\setminus\D$. 
Clearly, $f_\pm(\0)=0$, 
\begin{equation}\label{eq:L=}
	L_\pm(\x):=f'_\pm(\0)(\x)=\frac{x_1}{\E U_1}=\frac{x_1}{I(\th_0)}  
\end{equation}
for $\x=(x_1,x_2,x_3)\in\R^3$, and, in accordance with \eqref{eq:M^*}, the smoothness condition \eqref{eq:smooth} holds for some $\ep$ and $M_\ep$ in $(0,\infty)$ -- because, as was noted above, $\E U_1=I(\th_0)\in(0,\infty)$ and $\E R_1^*\in[0,\infty)$, and hence the denominator of the ratio in \eqref{eq:f_pm} is bounded away from $0$ for $\x=(x_1,x_2,x_3)$ in a neighborhood of $\0$. 

Next, let 
\begin{equation}\label{eq:V_i=}
	V_i:=(Z_i,U_i-\E U_i,R_i^*-\E R_i^*)
\end{equation}
for $i=1,\dots,n$, with $Z_i,U_i,R_i^*$ as defined in \eqref{eq:Z,R}. 
Then, by \eqref{eq:tsi,v_p,vsi_p}, \eqref{eq:L=}, and condition~\eqref{fisher}, 
for $f=f_\pm$, 
\begin{equation}\label{eq:tsi=}
\tsi=\sqrt{\frac{\E Z_1^2}{I(\th_0)^2}}=\frac1{\sqrt{I(\th_0)}}>0	
\end{equation}
and $v_3^3=\E\|V\|^3<\infty$ by conditions \eqref{M_2} and \eqref{M_3}. 
So, all the conditions of Theorem~\ref{th:nonlin} are satisfied for $f=f_\pm$.  

Moreover, by \eqref{eq:T_pm}, \eqref{eq:f_pm}, and \eqref{eq:V_i=},  
\begin{equation*}
	T_\pm=f_\pm(\bar V)  
\end{equation*}
on the event $G\setminus B$. 
So, 
by the inclusion relation in \eqref{eq:hath-th0=} \big(which holds on the event $G\setminus B=(G^\cc\cup B)^\cc$, where ${}^\cc$ denotes the complement\big) and \eqref{eq:tsi=}, inequality~\eqref{eq:f(S).iid} in Theorem~\ref{th:nonlin} implies 
\begin{equation*}
\begin{aligned}
 \P\Big(\sqrt{nI(\th_0)}\,(\hat\th-\th_0)\le z\Big)
 &\le\P\Big(\sqrt{nI(\th_0)}\,f_-(\bar V)\le z\Big)+\P(G^\cc\cup B) \\ 
 &\le\Phi(z)+\frac{\CC}{\sqrt{n}}+\P(G^\cc\cup B)
\end{aligned} 
\end{equation*}
and, quite similarly, 
\begin{equation*}\label{eq:lower}
\begin{aligned}
 \P\Big(\sqrt{nI(\th_0)}\,(\hat\th-\th_0)\le z\Big)
 &\ge\P\Big(\sqrt{nI(\th_0)}\,f_+(\bar V)\le z\Big)-\P(G^\cc\cup B) \\ 
 &\ge\Phi(z)-\frac{\CC}{\sqrt{n}}-\P(G^\cc\cup B), 
\end{aligned}  
\end{equation*}
for all real $z$. 
Note that $\P(G^\cc\cup B)=\P(G^\cc)+\P(G\cap B)$. It follows now by \eqref{eq:G} and \eqref{eq:P(G and B)} that 
\begin{equation}\label{eq:ub}
	\Big|\P\Big(\sqrt{nI(\th_0)}\,(\hat\th-\th_0)\le z\Big)-\Phi(z)\Big|\le\frac{\CC}{\sqrt{n}}+\P(|\hat\th-\th_0|>\de)
\end{equation}
for all real $z$. 
Quite similarly, but using \eqref{eq:f(S).iid.power} instead of \eqref{eq:f(S).iid}, one has 
\begin{equation}\label{eq:nub}
	\Big|\P\Big(\sqrt{nI(\th_0)}\,(\hat\th-\th_0)\le z\Big)-\Phi(z)\Big|\le\frac{\CC}{z^3\,\sqrt{n}}+\P(|\hat\th-\th_0|>\de)
\end{equation}
for $z$ as in \eqref{eq:z,iid}. 


Typically, given rather standard regularity conditions, the remainder term $\P(|\hat\th-\th_0|>\de)$ decreases exponentially fast in $n$ and thus is negligible as compared with the ``error'' term $\frac{\CC}{\sqrt{n}}$, and even with the ``error'' term $\frac{\CC}{z^3\,\sqrt{n}}$ -- under condition \eqref{eq:z,iid}. 
Some details on this can be found in the following section. 


\section{Exponentially small bounds on the remainder term \texorpdfstring{$\P(|\hat\th-\th_0|>\de)$}{} }
\label{remainder}

\subsection{Bounding the remainder: Log-concave case}
\label{log-conc}
In this subsection, suppose that the log-likelihood $\ell_x(\th)$ is concave in $\th\in\Th$, for each $x\in\XXX$. By condition \eqref{fisher}, $\E \ell''_X(\th_0)\ne0$. Hence, $\P\big(p_{\th_0+h}(X)\ne p_{\th_0}(X)\big)
=\P\big(\ell_X(\th_0+h)\ne\ell_X(\th_0)\big)>0$ for some $h\in(0,\de)$. 
The concavity of $\ell_x(\th)$ in $\th$ implies that of $\ell_\X(\th)$. So, if $\hat\th>\th_0+\de$, then $\ell_\X(\th_0+h)\ge\ell_\X(\th_0)$. Therefore, 
\begin{multline*}
	\P(\hat\th>\th_0+\de)\le\P\big(\ell_\X(\th_0+h)\ge\ell_\X(\th_0)\big)
	=\P\Big(\prod_{i=1}^n\sqrt{\frac{p_{\th_0+h}(X_i)}{p_{\th_0}(X_i)}}\ge1\Big) \\ 
	\le\E\prod_{i=1}^n\sqrt{\frac{p_{\th_0+h}(X_i)}{p_{\th_0}(X_i)}}=\la_+^n, 
\end{multline*}
where 
\begin{equation*}
	\la_+:=\E\sqrt{\frac{p_{\th_0+h}(X)}{p_{\th_0}(X)}}<\sqrt{\E\frac{p_{\th_0+h}(X)}{p_{\th_0}(X)}}=
	\sqrt{\E_{\th_0}\frac{p_{\th_0+h}(X)}{p_{\th_0}(X)}}=1;
\end{equation*}
the inequality here is an instance of a strict version of the Cauchy--Schwarz inequality, which holds because, as was noted, $\P\big(p_{\th_0+h}(X)\ne p_{\th_0}(X)\big)>0$. 
Quite similarly, $\P(\hat\th<\th_0-\de)\le\la_-^n$ for some $\la_-\in[0,1)$, and so, 
\begin{equation}\label{eq:log-conc}
	\P(|\hat\th-\th_0|>\de)\le2\la^n
\end{equation}
for $\la:=\max(\la_+,\la_-)\in[0,1)$. 

In particular, the condition of the concavity of the log-likelihood $\ell_x(\th)=\ln p_\th(x)$ in $\th$ is fulfilled in the important case when the densities $p_\th$ form an exponential family with $\th$ as the natural parameter,  
so that  
\begin{equation*}
 p_\th(x)=e^{\th g(x)-\psi(\th)}
\end{equation*}
for some function $\psi\colon\Th\to\R$ and 
all $\th\in\Th$ and $x\in\XXX$. Here, $g\colon\XXX\to\R$ is a  measurable function. Then necessarily $\psi(\th)=\ln\int_\XXX e^{\th g(x)}\mu(dx)$, which is convex in $\th$ -- because any mixture of log-convex functions is log-convex, as is well known -- see e.g.\ \cite[page~66, Theorem~5.4C]{keilson79}. So, $\ell_x(\th)=\ln p_\th(x)=\th g(x)-\psi(\th)$ is indeed concave in $\th$. 
In the case of multivariate exponential families, an exponentially decreasing bound of a form more complicated than that of the bound in \eqref{eq:log-conc} was given in \cite{kour84}. 

\subsection{Bounding the remainder: General case}
\label{remainder,general}

Upper bounds on the large-deviation probability $\P(|\hat\th-\th_0|>\de)$ that are exponentially decreasing in $n$ without the assumption of the concavity of the log-likelihood function were presented e.g.\ in \cite{radav81,radav83,mogul88,borovkov_MS,miao10}. 
However, the parameter space $\Th$ was assumed in \cite{radav81,radav83,borovkov_MS} to be bounded, whereas in \cite{miao10} the distributions $\P_\th$ were assumed to be subgaussian (cf. Theorems~2.1, 2.2, and 3.3 in \cite{miao10}). 
Conditions in \cite{mogul88} appear to be difficult to verify, including the strict positivity of the infimum of the rate function, needed for an actual exponential decrease. 

Related is the work \cite{ibr-radav}, containing a result on so-called moderate deviation probabilities for MLEs, which decrease slower than exponentially but still faster than any powers. So, such a result would be enough for our conclusions in Theorem~\ref{th:main} in the next section (cf.\ Remark~\ref{rem:slower} there), if it were not assumed in \cite{ibr-radav} (as in \cite{radav81,radav83,borovkov_MS}) that $\Th$ is bounded. 

Here we modify the method of \cite{borovkov_MS} to get rid of the condition that $\Th$ is bounded. 
Consider the (squared) Hellinger distance 
\begin{equation}\label{eq:H:=}
H(\th,\th_0):=\int_\XXX\big(\sqrt{p_\th}-\sqrt{p_{\th_0}}\big)^2\dd\mu
\end{equation}
between the probability measures $\P_\th$ and $\P_{\th_0}$. 
 
Assume now the following conditions: 
%
\begin{enumerate}
\item[(B)]
The set $\Th$ is a (possibly infinite) interval, and the Fisher information $I(\th)$ is well defined and satisfies the boundedness condition 	
\begin{equation}\label{eq:I<}
I(\th)\le c_1+c_2|\th-\th_0|^\al 	
\end{equation}
for some positive real constants $c_1,c_2,\al$ and all $\th\in\Th$. \big(If a point $\th$ in $\Th$ is an endpoint of the interval $\Th$, then $I(\th)$ is naturally understood in terms of the corresponding one-sided derivative of $p_\th(x)$ in $\th$.\big) 
	\item[$(\text{D}_0)$] 
	\label{H1} For each bounded neighborhood $U$ of $\th_0$, 
\begin{equation}\label{eq:H,bor}
	H(\th,\th_0)\OG(\th-\th_0)^2
\end{equation}
over  
all $\th\in U$. 
	\item[$(\text{D}_1)$] \label{H2} For some real constant $\ga>0$ and some bounded neighborhood $V$ of $\th_0$, 
\begin{equation}\label{eq:large th}
	J(\th,\th_0):=1-\tfrac12\,H(\th,\th_0)=\int_\XXX\sqrt{p_\th}\sqrt{p_{\th_0}}\dd\mu\O|\th-\th_0|^{-\ga}
\end{equation}
over  
all $\th\in\Th\setminus V$.  
\end{enumerate}
Here and in the sequel, for any two expressions $E_1>0$ and $E_2\ge0$ whose values depend on some variables, the relation $E_1\OG E_2$ and its equivalent $E_2\O E_1$ mean that $\sup(E_2/E_1)<\infty$, where the supremum is taken over the corresponding specified range of values of the variables. 
%

Conditions $(\text{D}_0)$ and $(\text{D}_1)$ may be referred to as distinguishability conditions:  $(\text{D}_0)$ means that the probability measures $\P_\th$ and $\P_{\th_0}$ are not too close to each other for $\th$ in a punctured neighborhood of $\th_0$, whereas 
$(\text{D}_1)$ implies that 
for $\th$ far away from $\th_0$, the probability measures $\P_\th$ and $\P_{\th_0}$ are almost mutually singular, and thus, easily distinguishable, at least in principle. 
 
\begin{remark}\label{rem:compact}
In the particular case when the parameter space $\Th$ is compact (or just bounded), condition $(\text{D}_1)$ 
trivially holds. 
Moreover, as shown in \cite[Section~31]{borovkov_MS}, if 
$\Th$ is compact and 
the Fisher information $I(\th)$ is continuous in $\th\in\Th$ and strictly positive for $\th\in\Th$, then \eqref{eq:H,bor} holds over all $\th\in\Th$. 
So, 
condition $(\text{D}_0)$ holds (whether the set $\Th$ is bounded or not) whenever the Fisher information $I(\cdot)$ is continuous and strictly positive on $\Th$. 
\end{remark} 


However, since $H(\th,\th_0)$ is always bounded from above by $2$, it is clear that condition \eqref{eq:H,bor} cannot possibly hold over all $\th\in\Th$ if the parameter space $\Th$ is unbounded. In such a case, we need to complement  condition $(\text{D}_0)$ by condition $(\text{D}_1)$,  
%
which latter appears to be natural, and it is indeed commonly satisfied. In particular, 
conditions $(\text{D}_0)$ and $(\text{D}_1)$ \big(as well as regularity conditions \eqref{diff}--\eqref{M_3}\big) -- 
hold if $p_\th$ is the density belonging to any one of the following families of probability distributions: 
\begin{enumerate}[(a)]
	\item \label{normal} $\mathrm{N}(\th,\si^2)$ -- with $\si>0$ known, $\Th=\R$, $H(\th,\th_0)=2-2\exp\big\{-\frac{(\th-\th_0)^2}{8\si^2}\big\}$;  
	\item $\mathrm{N}(\mu,\th^2)$ -- with $\mu>0$ known, $\Th=(0,\infty)$, $H(\th,\th_0)=2-2\sqrt{\frac{2\th\th_0}{\th^2+\th_0^2}}$;  
	\item $\mathrm{Exp}(\th)$ -- with $\Th=(0,\infty)$, $H(\th,\th_0)=2-\frac{4\sqrt{\th\th_0}}{\th+\th_0}$;  
	\item more generally, Weibull distributions $\mathrm{W}(k,\th)$ -- with  \\ 
	$p_\th(x)\equiv  
	\frac k\th(\frac x\th)^{k-1}e^{-(x/\th)^k}I\{x>0\}$, $k>0$ known, $\Th=(0,\infty)$, $H(\th,\th_0)=2-\frac{4(\th\th_0)^{k/2}}{\th^k+\th_0^k}$;  
	\item $\mathrm{Gamma}(\th,\be)$ -- with scale parameter $\be>0$ known, $\Th=(0,\infty)$, $H(\th,\th_0)=2-\frac{2\Ga((\th+\th_0)/2)}{\sqrt{\Ga(\th)\Ga(\th_0)}}$;  
	\item $\mathrm{Gamma}(\al,\th)$ -- with shape parameter $\al>0$ known, $\Th=(0,\infty)$, $H(\th,\th_0)=2-\frac{2^{1+\al}(\th\th_0)^{\al/2}}{(\th+\th_0)^\al}$;  
	\item $\mathrm{Poisson}(\th)$ -- with $\Th=(0,\infty)$, $H(\th,\th_0)=2-2e^{-(\sqrt\th-\sqrt\th_0)^2/2}$;  
	\item $\mathrm{Beta}(s\th,s(1-\th))$ -- with $s>0$ known, $\Th=(0,1)$, $H(\th,\th_0)
	= \\ 
	2-\frac{2 \BB\left(\frac{1}{2} s \left(\theta +\theta _0\right),\frac{1}{2} s
   \left(2-\theta -\theta _0\right)\right)}{\sqrt{\BB(s \theta ,s-s \theta )
   \BB\left(s \theta _0,s-s \theta _0\right)}}$, where $\BB(\cdot,\cdot)$ is the Beta function;  
	\item $\mathrm{Beta}(\al\th,\be\th)$ -- with $\al,\be>0$ known, $\Th=(0,\infty)$, $H(\th,\th_0)= \\ 
	2-\frac{2 \BB\left(\frac{1}{2} \alpha  \left(\theta +\theta _0\right),\frac{1}{2}
   \beta  \left(\theta +\theta _0\right)\right)}{\sqrt{\BB(\alpha  \theta ,\beta 
   \theta ) \BB\left(\alpha  \theta _0,\beta  \theta _0\right)}}$ \big(in this case, by Stirling's formula, 
condition $(\text{D}_1)$ holds 
with $\ga=1/4$
\big). 
\end{enumerate}
\qed

\medskip

Item \eqref{normal} above, concerning the normal location family, can be quite broadly generalized: 

\begin{proposition}\label{prop:shift}
Suppose that $(p_\th)_{\th\in\Th}$ is a location family over $\R$, so that $p_\th(x)=p(x-\th)$ for all $x\in\R$ and $\th\in\Th$, where $p$ is a pdf (with respect to the Lebesgue measure over $\R$). Suppose also that 
\begin{equation}\label{eq:p<}
	p(u)\O(1+|u|)^{-\al}
\end{equation}
for some real $\al>1$ and all real $u$. 
Then condition \emph{$(\text{D}_1)$} holds. 
\end{proposition}
Note that the restriction $\al>1$, together with \eqref{eq:p<}, implies the integrability of the nonnegative function $p$. 

\begin{proof}[Proof of Proposition~\ref{prop:shift}]
Without loss of generality, $\th_0=-\th$ and $\th>0$, so that $\th-\th_0=2\th>0$ and  
\begin{equation}
	J(\th,\th_0)=
	\int_\R\sqrt{p(x+\th)}\sqrt{p(x-\th)}\dd x=\int_{|x|\ge2\th}\cdots\;+\int_{|x|<2\th}\cdots. 
\end{equation}
Since $|x|\ge2\th$ implies $|x\pm\th|\OG|x|$, condition \eqref{eq:p<} yields 
\begin{equation}
	\int_{|x|\ge2\th}\cdots\O\int_{|x|\ge2\th}|x|^{-\al}\dd x\O\th^{1-\al}. 
\end{equation}
Since $0\le x<2\th$ implies $x+\th\ge\th$ and $-\th\le x-\th<\th$, condition \eqref{eq:p<} yields 
\begin{equation}
	\int_0^{2\th}\ldots\O\th^{-\al/2}\int_{-\th}^\th(1+|u|)^{-\al/2}\dd u\O\th^{-\al/2}\,\th^{0\vee(1-\al/2)} \ln\th 
\end{equation}
for (say) $\th\ge2$; the factor $\ln\th$ is actually needed here only in the case when $\al=2$. 
The integral $\int_{-2\th}^0\cdots$ can be bounded quite similarly. So, 
\begin{equation}
	\int_{|x|<2\th}\ldots\O\th^{-\al/2}\,\th^{0\vee(1-\al/2)} \ln\th
\end{equation}
for $\th\ge2$. 
Thus, $(\text{D}_1)$ holds for any $\ga\in(0,\ga_\al)$, where $\ga_\al:=\frac\al2-\big(0\vee(1-\frac\al2)\big)=\frac\al2\wedge(\al-1)>0$. 
\end{proof}

The problem concerning the possibility of a non-compact parameter space $\Th$ may be illustrated by the following simple example:

\begin{example}\label{ex:non-compact}
For $\th\in\Th=(-1,\infty)$, let 
$p_\th$ be the density (with respect to the Lebesgue measure on $\R$) of the normal distribution with mean $\mu(\th):=\frac\th{1+\th^2}$ and variance $\si^2(\th):=\frac{(1+\th)^3-\th}{1+\th^3}$, and let $\th_0=0$, so that $\th_0\in\Th^\circ=\Th$. Then 
for any two distinct $\th$ and $\tau$ in $\Th$ the equality $\mu(\tau)=\mu(\th)$ implies $\th\notin\{0,1\}$ and $\tau=1/\th>0$, whence $\si^2(\tau)\ne\si^2(\th)$. So,  
$p_\tau\ne p_\th$ for any two distinct $\th$ and $\tau$ in $\Th$. However, $\mu(\th)\underset{\th\to\infty}\longrightarrow0=\mu(0)$ and $\si^2(\th)\underset{\th\to\infty}\longrightarrow1=\si^2(0)$, so that $p_0$ is almost indistinguishable from $p_\th$ for large $\th$. More specifically, it is not hard to check that here 
\begin{equation*}
	J(\th,\th_0)=\int_\R\sqrt{p_\th(x)}\sqrt{p_0(x)}\dd x
	=\sqrt{\frac{2\si(\th)}{\si^2(\th)+1}} \exp\Big(-\frac{\mu(\th) ^2}{4 \si^2(\th)+4}\Big)  
	\underset{\th\to\infty}\longrightarrow1, 
\end{equation*}
so that this situation is excluded by condition \eqref{eq:large th}. 
\end{example}

Now we are well prepared to state the main result of this subsection: 

\begin{proposition}\label{prop:la^n}
Under 
conditions \emph{(B)}, \emph{$(\text{D}_0)$}, and \emph{$(\text{D}_1)$}, 
\begin{equation}\label{eq:bor}
	\P(|\hat\th-\th_0|>\de)\le c\,\la^n
\end{equation}
for some real constants $c>0$ and $\la\in[0,1)$ (depending on $\ga,c_0,\al,c_1,c_2$) and all natural $n$; cf.\ \eqref{eq:log-conc}.
\end{proposition}

Inequality \eqref{eq:bor} is similar to inequality (6) in \cite[Section~33.2, Theorem~3]{borovkov_MS}, with the following main differences. 

\begin{enumerate}[(i)] 
	\item It is assumed in \cite{borovkov_MS} that $\Th$ is compact, in addition to the assumption that $I(\th)$ is continuous in $\th\in\Th$ and strictly positive for $\th\in\Th$. Under these assumptions, condition $(\text{D}_0)$ is, not assumed, but derived in \cite{borovkov_MS}. As noted above, if the parameter space $\Th$ is compact, then condition 
$(\text{D}_1)$ is trivial. 
	\item As we do not assume that $\Th$ is compact (or even bounded), we need to control the behavior of log-likelihood $\ell_\X(\th)$ for $\th$ far from $\th_0$. This is done using condition $(\text{D}_1)$. 
	\item In \cite{borovkov_MS}, instead of condition (B) above, it is assumed that the Fisher information $I(\th)$ is just bounded over all $\th\in\Th$. However, mainly following the lines of proof in \cite{borovkov_MS}, one can see that the more general condition (B) suffices, given conditions 
$(\text{D}_0)$ and $(\text{D}_1)$. 
\end{enumerate}

For the readers' convenience here is 


\begin{proof}[Proof of Proposition~\ref{prop:la^n}]
Let 
\begin{equation}
	Z(u):=\frac{p_{\th_0+u}(\X)}{p_{\th_0}(\X)}=\prod_{i=1}^n \frac{p_{\th_0+u}(X_i)}{p_{\th_0}(X_i)}=\exp\{\ell_\X(\th_0+u)-\ell_\X(\th_0)\}, 
\end{equation}
where $p_\th(\X):=\prod_{i=1}^n p_\th(X_i)=\exp\ell_\X(\th)$ and 
$\ell_\X$ is the log-likelihood function, as defined in \eqref{eq:ell};  
here and subsequently in this proof, $u$ is a real number such that $\th_0+u\in\Th$.  

By conditions $(\text{D}_1)$ and $(\text{D}_0)$, there exist real $C_1>0$, 
\begin{equation}\label{eq:u_*>}
u_*>C_1^{1/\ga}\vee\de, 	
\end{equation}
and $C_0>0$ such that 
\begin{equation}\label{eq:Z^{1/2},1}
	\E Z(u)^{1/2}=\E_{\th_0} Z(u)^{1/2}=J(\th_0,\th_0+u)^n\le C_1^n u^{-n\ga}\quad\text{if }|u|>u_*
\end{equation}
and 
\begin{equation}\label{eq:Z^{1/2},0}
	\E Z(u)^{1/2}=\big(1-\tfrac12\,H(\th_0,\th_0+u)\big)^n
	\le (1-u^2/C_0)^n \le e^{-nu^2/C_0} \quad\text{if }|u|\le u_*. 
\end{equation} 
Note also that $\E Z(u)=1$. So, introducing 
\begin{equation}
	P(u):=Z(u)^{3/4}, 
\end{equation}
by the Cauchy--Schwarz inequality one has 
\begin{equation}\label{eq:EP}
	\E P(u)\le\sqrt{\E Z(u)\,\E Z(u)^{1/2}}=\sqrt{\E Z(u)^{1/2}}. 
\end{equation}
Further, $P'(u)=\frac34\,\ell'_\X(\th_0+u)Z(u)^{3/4}$, whence, again by the Cauchy--Schwarz inequality,
\begin{equation}\label{eq:EP'} 
\begin{aligned}
	\E|P'(u)|&\le\tfrac34\,\sqrt{\E\ell'_\X(\th_0+u)^2 Z(u)\,\E Z(u)^{1/2}} \\ 
	&=\tfrac34\,\sqrt{\E_{\th_0+u}\ell'_\X(\th_0+u)^2 \,\E Z(u)^{1/2}} \\ 
	&=\tfrac34\,\sqrt{nI(\th_0+u) \,\E Z(u)^{1/2}}. 
\end{aligned}
\end{equation}
For $u>\de$, one has $P(u)\le P(\de)+\int_{\Th\cap(\de,\infty)}|P'(t)|\,dt$. So, by \eqref{eq:EP}, \eqref{eq:Z^{1/2},0}, \eqref{eq:EP'}, \eqref{eq:I<}, \eqref{eq:u_*>}, and \eqref{eq:Z^{1/2},1},   
\begin{equation*}
	\E\sup_{u>\de}P(u)\le e^{-n\de^2/(2C_0)}+I_0+I_1=\la_*^n+I_0+I_1, 
\end{equation*}
where $\la_*:=e^{-\de^2/(2C_0)}\in(0,1)$, 
\begin{equation*}
	I_0:=\int_\de^{u_*}\sqrt{n(c_1+c_2u^\al)}\,e^{-nu^2/(2C_0)}\,du
	\O\int_\de^\infty \sqrt{nu^\al}\,e^{-nu^2/(2C_0)}\,du \O 
	\la_0^n      
\end{equation*}
for any fixed $\la_0\in(\la_*,1)$, 
and 
\begin{equation*}
	I_1:=\int_{u_*}^\infty \sqrt{n(c_1+c_2u^\al)C_1^n u^{-n\ga}}\,du\O\la_1^{n/2}  
\end{equation*}
for any fixed $\la_1\in(C_1/u_*^\ga,1)$ -- note that the latter interval is nonempty, in view of \eqref{eq:u_*>}.  
Thus, $\E\sup_{u>\de}P(u)\O\la^n$ for $\la:=\la_0\vee\sqrt{\la_1}\in(0,1)$. Quite similarly, $\E\sup_{u<-\de}P(u)\O\la^n$ and hence $\E\sup_{|u|>\de}P(u)\O\la^n$. 
So, 
\begin{equation*}
	\P(|\hat\th-\th_0|>\de)\le\P(\sup_{|u|>\de}Z(u)\ge Z(0))=\P(\sup_{|u|>\de}P(u)\ge1)\le\E\sup_{|u|>\de}P(u)\O\la^n, 
\end{equation*}
which completes the proof of Proposition~\ref{prop:la^n}. 
\end{proof}

\section{
Conclusion}\label{concl}

Inequalities \eqref{eq:ub} and \eqref{eq:nub} together with \eqref{eq:log-conc} and Proposition~\ref{prop:la^n} yield 

\begin{theorem}\label{th:main}
Suppose that conditions \eqref{fisher}, \eqref{M_2}, and \eqref{M_3} hold. Suppose also that either (i) the log-likelihood $\ell_x(\th)$ is concave in $\th\in\Th$, for each $x\in\XXX$, or (ii) conditions 
\emph{(B)}, \emph{$(\text{D}_0)$}, and \emph{$(\text{D}_1)$} 
hold. Then 
\begin{equation}\label{eq:ub,mle}
	\Big|\P\Big(\sqrt{nI(\th_0)}\,(\hat\th-\th_0)\le z\Big)-\Phi(z)\Big|\le\frac{\CC}{\sqrt{n}}
\end{equation}
for all real $z$, and  
\begin{equation}\label{eq:nub,mle}
	\Big|\P\Big(\sqrt{nI(\th_0)}\,(\hat\th-\th_0)\le z\Big)-\Phi(z)\Big|\le\frac{\CC}{z^3\,\sqrt{n}} 
\end{equation}
for $z$ as in \eqref{eq:z,iid}. Here, as before, each of the two instances of the symbol $\CC$ stands for a finite positive expression whose values do not depend on $n$ or $z$, in accordance with the last paragraph of Section~\ref{f(bar V)}.  
\end{theorem}

\begin{remark}\label{rem:slower}
It should be clear that the conditions assumed in the second sentence of Theorem~\ref{th:main} can be replaced by any other conditions that imply \eqref{eq:bor} for some real constants $c>0$ and $\la\in[0,1)$ not depending on $n$. 
Actually, a much weaker bound, of the form $c/n^2$, instead of the exponentially fast decreasing upper bound $c\la^n$ in \eqref{eq:bor}, will already suffice. 
\end{remark}

Theorem~\ref{th:main} can be extended to the more general case of $M$-estimators. 
Indeed, the condition that $p_\th$ is a pdf for $\th\ne\th_0$ is used in our proofs only in order to state that $\E_\th\ell'_X(\th)=0$ and $\E_\th \ell'_X(\th)^2=-\E_\th \ell''_X(\th)=I(\th)\in(0,\infty)$. In the case of $M$-estimators
, the corresponding conditions will have to be just assumed, with some other expressions in place of the Fisher information $I(\th)$, as it is done e.g.\ in \cite{pfanzagl71,pfanzagl73}, where 
uniform 
bounds of optimal order $O(1/\sqrt n)$ for $M$-estimators were obtained;    
$M$-estimators were referred to as minimum contrast estimates in \cite{michel-pfanzagl71,pfanzagl71,pfanzagl73}. We have chosen to restrict the consideration here to MLEs in order not to obscure the novelty elements in our result. 

The most significant novelty in our Theorem~\ref{th:main}, as compared with the results of \cite{michel-pfanzagl71,pfanzagl71,pfanzagl73}, is that, in addition to the uniform bound in \eqref{eq:ub,mle}, inequality \eqref{eq:nub,mle} in Theorem~\ref{th:main} also provides a nonuniform Berry--Esseen-type bound for MLEs in general, which latter appears to be the first such result in the literature -- except for the already mentioned special case considered recently in \cite{nonlinear-publ}. On the other hand, paper \cite{pfanzagl73} treats the  case of a multidimensional parameter $\th$. The uniform bound in \cite{michel-pfanzagl71} was of the form $O(\sqrt{\ln n}/\sqrt n)$, rather than of the optimal order $O(1/\sqrt n)$. 

Another notable distinction is that 
condition \cite[(1)]{michel-pfanzagl71} (the same as the corresponding conditions on page~73 in \cite{pfanzagl71} and on page~173 in \cite{pfanzagl73}) effectively reduces the consideration to the case when the parameter space $\Th$ is compact in $[-\infty,\infty]$. This obviates the need in a condition such as 
$(\text{D}_1)$, which is there to control the behavior of the likelihood $\ell_\X(\th)$ for large $|\th|$. However, as pointed out in \cite[page~75]{michel-pfanzagl71} concerning the main result there, the nonconstructive compactification condition used in \cite{michel-pfanzagl71,pfanzagl71,pfanzagl73} 
``gives no method for determining [the] value [of the constant in the Berry--Esseen-type bound] for a given family of probability measures.'' 

The problem of controlling the likelihood over far-away zones of a non-compact parameter space $\Th$ was  illustrated in Example~\ref{ex:non-compact}, where the ``bad'' 
situation was excluded by condition \eqref{eq:large th}. 
That same situation was also excluded by the mentioned compactification condition in \cite{michel-pfanzagl71,pfanzagl71,pfanzagl73} -- with $f_\th=-\ln p_\th$ for $\th\in\overline\Th=[-1,\infty]$ and $\mu(\infty):=\lim_{\th\to\infty}\mu(\th)=0=\mu(0)$ and variance $\si^2(\infty):=\lim_{\th\to\infty}\si^2(\th)=1=\si^2(0)$. 

As was pointed out, the method of the present paper is based on the general Berry--Esseen bounds for the multivariate delta method obtained in \cite{nonlinear-publ}, which were apllied here via the bracketing argument delineated in Section~\ref{bracketing}. As such, this method is quite different from the methods in \cite{michel-pfanzagl71,pfanzagl71,pfanzagl73}, specialized to deal with MLEs. 
Partly because of this difference in the methods, there are many differences between the conditions in \cite{michel-pfanzagl71,pfanzagl71,pfanzagl73} and those in the present paper. Most of these differences -- apart from the ones discussed above -- are rather minor. 
Since the result of \cite{pfanzagl71} is apparently the closest to ours in the literature, let us further discuss the regularity conditions in \cite{pfanzagl71}, in comparison with ours, in some detail:  

\begin{remark}\label{rem:compare}
Condition \eqref{diff} in the present paper can be replaced by the condition that $p_\th>0$ everywhere on $\XXX$. The latter condition is necessary in order for $\ell_x(\th)=\ln p_\th(x)$ to be defined for all $x\in\XXX$; cf.\ the first paragraph on page~83 in \cite{pfanzagl71}. 

Our condition \eqref{fisher} follows, by Remark~\ref{rem:fisher}, from regularity conditions (iv), (v)(a), (vi) on pages~83--84 in \cite{pfanzagl71} -- for $f_\th:=-\ell_\th$.   

Next, condition \eqref{M_2} follows from \cite[(vi)]{pfanzagl71}. Here and in the rest of this remark, the lower-case Roman numerals and letters in parentheses refer to the regularity conditions on pages~83--84 in \cite{pfanzagl71} -- again for $f_\th:=-\ell_\th$.   

Next, condition \eqref{M_3} is, in main, a bit stronger than \cite[(viii)]{pfanzagl71}. Of course, condition \eqref{M_3} can be 
relaxed, for the price of making it more complicated.  

By Remark~\ref{rem:compact}, our condition $(\text{D}_0)$ will hold if the Fisher information $I(\cdot)$ is continuous and strictly positive on $\Th$, for which conditions (ix) and (v)(a), respectively, in \cite{pfanzagl71} will be more than enough. 

Next, our condition $(\text{D}_1)$, to control the behavior of the likelihood $\ell_\X(\th)$ for large $|\th|$, was already discussed at length, versus the compactification condition used in \cite{michel-pfanzagl71,pfanzagl71,pfanzagl73}. 

In the case when $\Th$ is compact, for our condition (B) to hold, either one of regularity conditions (vi)(a) or (vi)(b) in \cite{pfanzagl71} will be more than enough. More generally, condition (B) together with condition $(\text{D}_1)$ replace the just mentioned compactification condition in \cite{michel-pfanzagl71,pfanzagl71,pfanzagl73}.   

In this paper, no explicit analogues of regularity conditions (i), (ii), (iii), (vii) of \cite{pfanzagl71} are imposed. 
\end{remark}

So, quite predictably, neither our conditions imply those in \cite{michel-pfanzagl71,pfanzagl71,pfanzagl73}, nor vice versa.  
However, our conditions appear to be a bit simpler and more explicit overall than those in \cite{michel-pfanzagl71,pfanzagl71,pfanzagl73}. It should also be mentioned that in \cite{michel-pfanzagl71,pfanzagl71} both the relevant conditions and the corresponding results are stated uniformly over compact subsets of $\Th$. Of course, a similar modification of our conditions and results can be done. 

\appendix

\section{On condition (1.1) 
} \label{append}

In this appendix it will be shown that, under general regularity conditions, \eqref{eq:q(th)} (or even a relaxed version of it) implies that the family of densities $(p_\th)$ is a one-parameter exponential one; thus, condition \eqref{eq:q(th)} is quite restrictive. 
This allows one to give any number of examples where Theorem~\ref{th:main} of the present paper is applicable, whereas \cite[Theorem~3.16]{nonlinear-publ} is not. Here we shall use a result of \cite{ferguson62}, which states that a location family can be a one-parameter exponential family only in two kinds of exceptional cases, when the densities are either normal or certain ``exponential-Gamma'' ones.    

\begin{proposition}\label{prop:exp}
Let us write here $\ell(x,\th)$ in place of $\ell_x(\th)$. 
Assume the following regularity conditions (cf.\ conditions \eqref{diff}--\eqref{M_3} on page~\pageref{diff}). 
\begin{enumerate}[(i)]
	\item \label{Th open} The parameter space $\Th$ is an open interval in $\R$. 
	\item \label{XXX_0} The set $\XXX_{>0}$ is the same as $\XXX$, which latter is an open interval in $\R$. 
	\item \label{mu=Leb} The measure $\mu$ is the Lebesgue measure on $\XXX$. 
	\item \label{exp-smooth} For $(x,\th)\in\XXX\times\Th$, there exist continuous partial derivatives of $\ell(x,\th)$ in $x$ and $\th$ of total order $4$.  
		\item \label{ell''} The Fisher information $I(\th)=-\E_\th\ell''_{\th\th}(X,\th)
		$ is everywhere finite, nowhere zero, and continuous in $\th\in\Th$. 
		\item \label{ell'' alt} Either 
		\begin{enumerate}
	\item \label{ell'' alt1} $\E_\th\sup\limits_{\tau\in[\th-\vp_\th,\th+\vp_\th]}|\ell'''_{\th\th\th}(X,\tau)|<\infty$ for all $\th\in\Th$ and some real $\vp_\th>0$ or 
	\item \label{ell'' alt2} $\E_\th|X|^{\de_\th}\sup\limits_{\tau\in[\th-\vp_\th,\th+\vp_\th]}|\ell''_{\th\th}(X,\tau)|<\infty$ for all $\th\in\Th$ and some real $\de_\th>0$ and $\vp_\th>0$.  
\end{enumerate}
		\item \label{ell'''} The expected values $\E_\th\ell'''_{\th\th\th}(X,\th)$ and $\E_\th\ell'''_{x\th\th}(X,\th)
		$ exist in $\R$ and are 
		continuous in $\th\in\Th$. 
		\item \label{ell''' alt} Either 
		\begin{enumerate}
	\item \label{ell''' alt1} $\E_\th\sup\limits_{\tau\in[\th-\vp_\th,\th+\vp_\th]}|\ell''''_{\th\th\th\th}(X,\tau)|<\infty$ for all $\th\in\Th$ and some real $\vp_\th>0$ or 
	\item \label{ell''' alt2} $\E_\th|X|^{\de_\th}\sup\limits_{\tau\in[\th-\vp_\th,\th+\vp_\th]}|\ell'''_{\th\th\th}(X,\tau)|<\infty$ for all $\th\in\Th$ and some real $\de_\th>0$ and $\vp_\th>0$.  
\end{enumerate}
	%
	\item \label{hat th} For each $\xx\in\XXX$ the corresponding MLE value $\hat\th_n(\xx)$ is in $\Th$, so that \break 
	$\sum_{i=1}^n\ell'_\th(x_i,\hat\th_n(\xx))=0$ for all $\xx\in\R^n$.  
		\item \label{consist} The MLE value $\hat\th_n(\X)$ is consistent: for each $\th\in\Th$, $\hat\th_n(\X)$ converges to $\th$ in $\P_\th$-probability.  
\end{enumerate}
Assume finally   
the following relaxed version of the condition \eqref{eq:q(th)} (cf.\ \cite[(3.27)]{nonlinear-publ}):  
there exist a twice continuously differentiable 
function $q\colon\Th\to\R$ with $q'(\th)\ne0$ for all $\th\in\Th$, a Borel function $g\colon\XX\to\R$, and a sequence $(E_n)_{n=1}^\infty$ such that $E_n$ is a Borel subset of $\R^n$ for 
each natural $n$,   
\begin{equation}\label{eq:notin E_n}
 \P_\th(\X\notin E_n)\underset{n\to\infty}\longrightarrow0
\end{equation}
for each $\th\in\Th$,  
and for each natural $n$ and each point $\xx=(x_1,\dots,x_n)\in E_n$ the MLE value $\hat\th_n(\xx)$ is in $\Th$ and satisfies the condition 
\begin{equation}\label{eq:q(th),s}
 q\big(\hat\th_n(\xx)\big)=\frac1n\,\sum_{i=1}^n g(x_i).   
\end{equation} 

Then the family of densities $(p_\th)_{\th\in\Th}$ is an exponential one; that is,
\begin{equation}\label{eq:exp-fam}
	\ell(x,\th)=w(\th)T(x)+c(\th)+h(x)
\end{equation}
for some functions $w$, $T$, $c$, and $h$, and all $(x,\th)\in\XXX\times\Th$.   
\end{proposition}


The proof of Proposition~\ref{prop:exp} 
will be preceded by a lemma, which may be of independent interest. 
To state the lemma, we need some notation:  
For any $\j=(j_1,\dots,j_n)\in\{1,2\}^n$ and any two points $\y\s1=(y\ss11,\dots,y\ss n1)$ and $\y\s2=(y\ss12,\dots,y\ss n2)$ in $\R^n$, let $\y\s\j:=(y\ss1{j_1},\dots,y\ss n{j_n})$. So, the set 
\begin{equation}
	V(\y\s1,\y\s2):=\{\y\s\j\colon\j\in\{1,2\}^n\}
\end{equation}
is the set of all vertices of the box 
$\prod_{i=1}^n[y\ss i1\wedge y\ss i2,\;y\ss i1\vee y\ss i2]$, 
with the edges parallel to the coordinate axes in $\R^n$, and $\y\s1$ and $\y\s2$ are two opposite vertices of the box. 
Now we can state the mentioned lemma: 

\begin{lemma}\label{lem:box}
Let $\ZZ:=(Z_1,\dots,Z_n)$, where $Z_1,\dots,Z_n$ are independent real-valued r.v.'s, with respect to a probability measure $\P$. 
Suppose that the r.v.'s $Z_i$ are continuous, in the sense that $\P(Z_i=z)=0$ for all real $z$ and $i=1,\dots,n$. 
Let $E\subseteq\R^n$ be any Borel set such that $\P(\ZZ\in E)>0$. Then there exist points $\y\s1=(y\ss11,\dots,y\ss n1)$ and $\y\s2=(y\ss12,\dots,y\ss n2)$ in $\R^n$ such that $V(\y\s1,\y\s2)\subseteq E$ and 
$y\ss i1\ne y\ss i2$ for all $i=1,\dots,n$. 
\end{lemma}

It is well known and easy to see that the support of a random vector with independent coordinates is a product set. Lemma~\ref{lem:box} states that, moreover, any Borel subset of $\R^n$ in which such a random vector lies with nonzero probability contains a nontrivial product set. 

It is also easy to see that neither the condition of independence nor that of continuity in Lemma~\ref{lem:box} can be dropped. For respective counterexamples, consider (say) (i) the random vector $(Z,\dots,Z)$ in $\R^n$ with $E$ being the set of all vectors of the form $(z,\dots,z)$ in $\R^n$, where $Z$ is any continuous r.v.\ and (ii) the random vector $(Z_1,\dots,Z_n)$ with $E=\{0,1\}^n\setminus\{(0,\dots,0)\}$, where $Z_1,\dots,Z_n$ are independent r.v.'s such that $\P(Z_i=0)=\P(Z_i=1)=1/2$ for all $i=1,\dots,n$. 

\begin{proof}[Proof of Lemma~\ref{lem:box}]
Since $\P(\ZZ\in E)>0$ and the distribution of a random vector in $\R^n$ is a regular Borel probability measure, there is a box $B=[a_1,b_1]\times\cdots\times[a_n,b_n]\subset\R^n$ such that \begin{equation}\label{eq:reg}
	\P(\ZZ\in B\cap E)>(1-2^{-n})\P(\ZZ\in B)>0. 
\end{equation}
Let now $\Y=(Y_1,\dots,Y_n)$ be any random vector in $\R^n$ such that 
\begin{equation}
	\P(\Y\in A)=\frac{\P(\ZZ\in B\cap A)}{\P(\ZZ\in B)}
\end{equation}
for all Borel sets $A\subseteq\R^n$; clearly, such a random vector $\Y$ exists. Then it is easy to see that the r.v.'s $Y_1,\dots,Y_n$ are independent and continuous,  
and 
\eqref{eq:reg} can be rewritten as 
\begin{equation}\label{eq:regul}
	\P(\Y\in 
	E)>1-2^{-n}. 
\end{equation} 
Let $\Y\s1$ and $\Y\s2$ be independent copies of the random vector $\Y$. 
Let the random vectors $\Y\s\j$ and the random set $V(\Y\s1,\Y\s2)$ in $\R^n$ be defined similarly to $\y\s\j$ and $V(\y\s1,\y\s2)$, but based on the random vectors $\Y\s1$ and $\Y\s2$ rather than non-random vectors $\y\s1$ and $\y\s2$. 
Consider the r.v. 
\begin{equation}
	N:=\sum_{\j\in\{1,2\}^n}\ii{\Y\s\j\in 
	E}, 
\end{equation}
where $\ii\cdot$ denotes the indicator function. 
Note that for each $\j\in\{1,2\}^n$ the random vector $\Y\s\j$ equals $\Y$ in distribution. 
So, in view of \eqref{eq:regul},  
\begin{equation}\label{eq:EN>}
	\E N=2^n\P(\Y\in 
	E)>2^n-1;
\end{equation}
here, the expectation $\E$ is of course with respect to the probability measure $\P$. 

It follows that $\P(N=2^n)=\P(N>2^n-1)
>0$. 
Thus,
\begin{equation}
\begin{aligned}
	0<\P(N=2^n)&
	=\P(\Y\s\j\in E\ \; \forall\j\in\{1,2\}^n) \\ 
	&=\P\big(V(\Y\s1,\Y\s2)\subseteq E\big) \\ 
	&=\P\big(V(\Y\s1,\Y\s2)\subseteq E\text{ and } Y\ss i1\ne Y\ss i2\ \forall i\in\{1,\dots,n\}\big); 
\end{aligned}	
\end{equation}
the last equality here holds because the r.v.'s $Y_i$ (as well as their independent copies $Y\ss i1$ and $Y\ss i2$) are continuous. 
Now the conclusion of Lemma~\ref{lem:box} immediately follows.  
\end{proof}

\begin{proof}[Proof of Proposition~\ref{prop:exp}]
Let 
\begin{equation}\label{eq:F}
	F(\xx,\th):=\sum_{i=1}^n\ell'_\th(x_i,\th).  
\end{equation}
In this proof, by default $\xx=(x_1,\dots,x_n)$ and $\th$ denote arbitrarily elements of $\XXX^n$ and $\Th$, respectively. 
By condition \eqref{ell''} and the law of large numbers, w.l.o.g.  
\begin{equation}\label{eq:ne0}
F'_\th(\xx,\th)=\sum_{i=1}^n\ell''_{\th\th}(x_i,\th)\ne0\text{ for }\xx\in E_n;  	
\end{equation}
otherwise, decrease the set $E_n$ by a set (say $H_n$) of negligible probability -- such that $\P_\th(\X\in H_n)\to0$ as $n\to\infty$. 
Here and in the rest of this proof, the lower-case Roman numerals and letters a and b in parentheses refer to the conditions in the statement of Proposition~\ref{prop:exp}, or parts (a) and (b) of those conditions. 

Moreover, again by the regularity of the distribution of $\X$, the set $E_n$ may be replaced by a large enough compact subset of $E_n$. So, w.l.o.g.\ the set $E_n$ is compact. 
Further, that compact set $E_n$ can be replaced by its intersection with the support set 
of the ``restriction'' of the distribution of $\X$ to $E_n$ defined by the formula $\mathcal{B}(\R^n)\ni  A\mapsto\P_\th(\X\in A\cap E_n)$, where $\mathcal{B}(\R^n)$ denotes the Borel sigma-algebra over $\R^n$; by condition \eqref{XXX_0}, the mentioned support set does not depend on $\th$. Thus, w.l.o.g.\ the set $E_n$ is compact (and hence closed), and for each $\xx\in E_n$ and each neighborhood $U$ of $\xx$ one has $\P_\th(\X\in U\cap E_n)>0$, whence,  
by Lemma~\ref{lem:box}, there exist points $\y\s1=(y\ss11,\dots,y\ss n1)$ and $\y\s2=(y\ss12,\dots,y\ss n2)$ in $\R^n$ such that $V(\y\s1,\y\s2)\subseteq U\cap E_n$ and 
$y\ss i1\ne y\ss i2$ for all $i=1,\dots,n$. 
So, by condition \eqref{eq:q(th),s} (assumed to hold for $\xx\in E_n$), 
\begin{equation}\label{eq:fin diff}
	q\big(\hat\th_n(\y\s{\j_{1,1}})\big)
	-q\big(\hat\th_n(\y\s{\j_{1,2}})\big)
	-q\big(\hat\th_n(\y\s{\j_{2,1}})\big)
	+q\big(\hat\th_n(\y\s{\j_{2,2}})\big)=0, 
\end{equation}
where $\j_{j_1,j_2}:=(j_1,j_2,1,\dots,1)\in\{1,2\}^n$ for $(j_1,j_2)\in\{1,2\}^2$. 
Taking now smaller and smaller neighborhoods $U$ of the point $\xx$, one will make all the four points $\y\s{\j_{j_1,j_2}}$ with $(j_1,j_2)\in\{1,2\}^2$ converge to $\xx$. 
Dividing both sides of \eqref{eq:fin diff} by $(y\ss11-y\ss12)(y\ss21-y\ss22)$, in the limit we will have 
\begin{equation}\label{eq:d x1,x2}
	\frac{\partial^2}{\partial x_1\partial x_2}q\big(\hat\th_n(\xx)\big)=0 \text{ for }\xx\in E_n; 
\end{equation}
note that the second-order partial derivative in \eqref{eq:d x1,x2} exists 
because (i) the function $q$ was assumed to be twice continuously differentiable and (ii) $\hat\th_n(\xx)$ is twice continuously differentiable by 
the implicit function theorem, in view of conditions \eqref{exp-smooth} 
and \eqref{eq:ne0}.  

In fact, by \eqref{eq:F} and \eqref{hat th}, $F(\xx,\hat\th(\xx))=0$ for all $x\in\XXX$, where $\hat\th(\xx):=\hat\th_n(\xx)$. 
Differentiating the identity $F(\xx,\hat\th(\xx))=0$ in $x_1$ and $x_2$, we have 
\begin{equation}\label{eq:hat th'}
F'_{x_i}+F'_\th\,\hat\th'_{x_i}=0 	
\end{equation}
for $i=1,2$, and then
\begin{equation}\label{eq:F'',F'}
	F''_{x_1,\th}\,\hat\th'_{x_2}+F''_{x_2,\th}\,\hat\th'_{x_1}
	+F''_{\th\th}\,\hat\th'_{x_1}\hat\th'_{x_2}+F'_\th\,\hat\th''_{x_1,x_2}=0;
\end{equation}
here the argument $(\xx,\hat\th(\xx))$ is omitted for brevity everywhere, and it is taken into account that, by \eqref{eq:F}, $F''_{x_1,x_2}=0$. 
On the other hand, by \eqref{eq:d x1,x2}, 
\begin{equation}\label{eq:q'',q'}
	q''(\hat\th)\,\hat\th'_{x_1}\hat\th'_{x_2}+q'(\hat\th)\,\hat\th''_{x_1,x_2}=0\text{ for }\xx\in E_n.  
\end{equation}
Multiplying equations \eqref{eq:F'',F'} and \eqref{eq:q'',q'} respectively by $q'(\hat\th)$ and $F'_\th$ and then subtracting one of the resulting equations from the other, we eliminate $\hat\th''_{x_1,x_2}$:  
\begin{equation}
	q'(\hat\th)\,(F''_{x_1,\th}\,\hat\th'_{x_2}+F''_{x_2,\th}\,\hat\th'_{x_1})
	+\big(q'(\hat\th)F''_{\th\th}-q''(\hat\th)F'_\th\big)\hat\th'_{x_1}\hat\th'_{x_2}=0  
\end{equation}
for $\xx\in E_n$. 
Multiply now the latter equation by $(F'_\th)^2/n$ and use \eqref{eq:hat th'} to eliminate $\hat\th'_{x_1}$ and $\hat\th'_{x_2}$:  
\begin{equation}\label{eq:main ident}
	q'(\hat\th)\,(F''_{x_1,\th}\,F'_{x_2}+F''_{x_2,\th}\,F'_{x_1})\frac{F'_\th}n
	+\Big(q'(\hat\th)\frac{F''_{\th\th}}n-q''(\hat\th)\frac{F'_\th}n\Big)F'_{x_1}F'_{x_2}=0 
\end{equation}
for $\xx\in E_n$. 
In the latter equation, replace $\xx$ by $\X$ and let $n\to\infty$. 
Then 
\begin{equation}\label{eq:q',q''}
 q'(\hat\th_n(\X))\tto{n\to\infty} q'(\th)\quad\text{and}\quad q''(\hat\th_n(\X))\tto{n\to\infty} q''(\th)  	
\end{equation}
in $\P_\th$-probability, by condition \eqref{consist} and because the function $q$ was assumed to be twice continuously differentiable. 
Similarly, by conditions \eqref{exp-smooth} and \eqref{consist}, for $i=1,2$ 
\begin{equation}\label{eq:F'_th,F'_x,th}
\begin{aligned}
 F'_{x_i}(\X,\hat\th_n(\X))&=\ell''_{\th x}(X_i,\hat\th_n(\X))\tto{n\to\infty} \ell''_{\th x}(X_i,\th), \\ 
 F'_{x_i,\th}(\X,\hat\th_n(\X))&=\ell'''_{x\th\th}(X_i,\hat\th_n(\X))\tto{n\to\infty} \ell'''_{x\th\th}(X_i,\th)
\end{aligned}   	
\end{equation}
in $\P_\th$-probability. 

Fix for a moment any $\th\in\Th$ and take any $\vp\in(0,\vp_\th)$, where $\vp_\th$ is as in condition \eqref{ell'' alt}. 
By the definition \eqref{eq:F}, 
on the event $\{|\hat\th_n(\X)-\th|\le\vp\}$ one has 
\begin{equation}\label{eq:F'th=}
	\frac{F'_\th(\X,\hat\th_n(\X))}n=G_{1,n}+G_{2,n}, 
\end{equation}
where	
\begin{equation}\label{eq:G_1}
	G_{1,n}:=\frac1n\,\sum_{i=1}^n\ell''_{\th\th}(X_i,\th)\underset{n\to\infty}\longrightarrow
	\E_\th\ell''_{\th\th}(X,\th)=-I(\th)
\end{equation} 
in $\P_\th$-probability, 
by \eqref{ell''} and the law of large numbers, and 
\begin{equation}\label{eq:G_2}
 	|G_{2,n}|\le\tilde G_{2,n}:=\frac1n\,\sum_{i=1}^n W_i, 
 	\quad\text{with}\quad 
 	W_i:=\sup_{\tau\in[\th-\vp,\th+\vp]}|\ell''_{\th\th}(X_i,\tau)-\ell''_{\th\th}(X_i,\th)|. 
\end{equation}
Next, $0\le W_i\le\vp V_{3,i}$, where $V_{3,i}:=\sup_{\tau\in[\th-\vp_\th,\th+\vp_\th]}|\ell'''_{\th\th\th}(X_i,\tau)|$. 

If the alternative \eqref{ell'' alt1} of condition \eqref{ell'' alt} holds, so that $c_{3,\th}:=\E_\th V_{3,1}<\infty$, then, again by the law of large numbers, 
$\tilde G_{2,n}\underset{n\to\infty}\longrightarrow\vp c_{3,\th}$ in $\P_\th$-probability. Since $\vp>0$ can be made arbitrarily small, it follows, in view of condition~\eqref{consist} and relations \eqref{eq:F'th=}, \eqref{eq:G_1}, and \eqref{eq:G_2} (which hold on the event $\{|\hat\th_n(\X)-\th|\le\vp\}$), that  
\begin{equation}\label{eq:F'th to}
	\frac{F'_\th(\X,\hat\th_n(\X))}n\tto{n\to\infty}-I(\th) 
\end{equation}
in $\P_\th$-probability -- when the alternative \eqref{ell'' alt1} of condition \eqref{ell'' alt} holds. 

Otherwise, the alternative \eqref{ell'' alt2} of condition \eqref{ell'' alt} must hold. Then take any real $A>0$ and write 
\begin{equation}\label{eq:W}
W_i\le w_{1,\vp,A}+W_{2,i,\vp,A}, 	
\end{equation}
where 
\begin{equation}
\begin{aligned}
	w_{1,\vp,A}&:=\sup_{|x|\le A,\tau\in[\th-\vp,\th+\vp]}|\ell''_{\th\th}(x,\tau)-\ell''_{\th\th}(x,\th)|, \\ 
	W_{2,i,\vp,A}&:=2\sup_{\tau\in[\th-\vp,\th+\vp]}|\ell''_{\th\th}(X_i,\tau)|\ii{|X_i|>A}. 
\end{aligned}	
\end{equation}
By \eqref{exp-smooth}, $\ell''_{\th\th}(x,\th)$ in continuous and hence uniformly continuous in $(x,\th)$ in any compact subset of the set $\XXX\times\Th$. So, 
\begin{equation}\label{eq:w_1}
	w_{1,\vp,A}\tto{\vp\downarrow0}0. 
\end{equation}
Next, again by the law of large numbers,  
\begin{equation}\label{eq:W_2i}
	\frac1n\,\sum_{i=1}^n W_{2,i,\vp,A}\tto{n\to\infty}\E_\th W_{2,1,\vp,A}
\end{equation}
in $\P_\th$-probability. 
On the other hand, $\E_\th W_{2,1,\vp,A}\le 2c_{2,\th}/A^{\de_\th}\tto{A\to\infty}0$, where $c_{2,\th}:=\E_\th|X|^{\de_\th}\sup_{\tau\in[\th-\vp,\th+\vp]}|\ell''_{\th\th}(X,\tau)|<\infty$ 
by the alternative \eqref{ell'' alt2} of condition \eqref{ell'' alt}. 
Therefore, in view of \eqref{eq:F'th=}, \eqref{eq:G_1}, \eqref{eq:G_2}, \eqref{eq:W}, \eqref{eq:w_1}, \eqref{eq:W_2i}, and condition \eqref{consist}, \eqref{eq:F'th to} holds as well under the alternative \eqref{ell'' alt2} of condition \eqref{ell'' alt}. 

Quite similarly -- but using condition \eqref{ell''' alt} instead of \eqref{ell'' alt} -- one verifies that 
\begin{equation}\label{eq:F''th to}
	\frac{F''_\th(\X,\hat\th_n(\X))}n\tto{n\to\infty}\E_\th\ell'''_{\th\th\th}(X,\th)  
\end{equation}
in $\P_\th$-probability. 
 
Recall now identity \eqref{eq:main ident}; make there the limit ``substitutions'' in accordance with \eqref{eq:q',q''}, \eqref{eq:F'_th,F'_x,th}, \eqref{eq:F'th to}, and \eqref{eq:F''th to}; finally, recall condition  
\eqref{mu=Leb} -- to conclude that 
\begin{equation}\label{eq:ell'',ell'''}
	\ell'''_{x\th\th}(x_1,\th)\ell''_{x\th}(x_2,\th)
	+\ell'''_{x\th\th}(x_2,\th)\ell''_{x\th}(x_1,\th)
	=\psi(\th)\ell''_{x\th}(x_1,\th)\ell''_{x\th}(x_2,\th)  
\end{equation}
for each $\th\in\Th$ and 
(Lebesgue-)almost all $(x_1,x_2)\in\XXX^2$, 
where 
\begin{equation}\label{eq:psi}
	\psi(\th):=-\frac{\E_\th\ell'''_{\th\th\th}(X,\th)}{I(\th)}-\frac{q''(\th)}{q'(\th)}. 
\end{equation}
So, by condition \eqref{exp-smooth}, identity \eqref{eq:ell'',ell'''} holds for all $\th\in\Th$ and all $(x_1,x_2)\in\XXX^2$. 

Consider the set 
\begin{equation*}
	\Th_0:=\{\th\in\Th\colon\ell''_{x\th}(x,\th)=0\ \forall x\in\XXX\}. 
\end{equation*}
If $\Th_0$ contains a nonempty open interval $J$, then integration of the differential equation $\ell''_{x\th}(x,\th)=0$ yields $\ell(x,\th)=c(\th)+h(x)$ and hence $p_\th(x)=e^{c(\th)}e^{h(x)}$ for some functions $c$ and $h$ and all $(x,\th)\in\XXX\times J$. Since $\int_\XXX p_\th(x)\mu(\dd x)=1$ for all $\th$, it follows that the probability density function $p_\th$ is the same for all $\th\in J$, which contradicts the assumption in \eqref{ell''} that the Fisher information is nowhere zero. 

Thus, the set $\Th\setminus\Th_0$ is dense in $\Th$. 
Fix for a moment any $\th\in\Th\setminus\Th_0$. If there is some $x_1\in\XXX$ such that $\ell''_{x\th}(x_1,\th)=0$ but $\ell'''_{x\th\th}(x_1,\th)\ne0$, then, by \eqref{eq:ell'',ell'''}, $\ell''_{x\th}(x_2,\th)=0$ for all $x_2\in\XXX$, which contradicts the condition 
$\th\in\Th\setminus\Th_0$. So, we have the implication 
\begin{equation}\label{eq:implic}
	\ell''_{x\th}(x,\th)=0\implies\ell'''_{x\th\th}(x,\th)=0
\end{equation}
for all $x\in\XXX$ and $\th\in\Th\setminus\Th_0$. 

Introduce now the set $S_\th:=\{x\in\XXX\colon\ell''_{x\th}(x,\th)\ne0\}$. 
For any $x_1$ and $x_2$ in $S_\th$, equality \eqref{eq:ell'',ell'''} can be rewritten as 
\begin{equation}\label{eq:ell'''/ell''}
	\frac{\ell'''_{x\th\th}(x_1,\th)}{\ell''_{x\th}(x_1,\th)}
	+\frac{\ell'''_{x\th\th}(x_2,\th)}{\ell''_{x\th}(x_2,\th)} 
	=\psi(\th),  
\end{equation}
which shows that $\frac{\ell'''_{x\th\th}(x,\th)}{\ell''_{x\th}(x,\th)}=\frac12\,\psi(\th)$ and hence 
\begin{equation}\label{eq:ident}
	\ell'''_{x\th\th}(x,\th)=\frac12\,\psi(\th)\ell''_{x\th}(x,\th), 
\end{equation}
for all $\th\in\Th\setminus\Th_0$ and $x\in S_\th$. 
Actually, in view of \eqref{eq:implic}, identity \eqref{eq:ident} holds for all $\th\in\Th\setminus\Th_0$ and $x\in\XXX$. 
Recalling now that the set $\Th\setminus\Th_0$ is dense in $\Th$ and using the smoothness condition \eqref{exp-smooth}, we conclude that \eqref{eq:ident} holds for all $\th\in\Th$ and $x\in\XXX$.
 
Let now $s(\th):=\exp(-\frac12\,\int\psi(\th)\dd\th)$, where $\int\psi(\th)\dd\th$ denotes an arbitrary antiderivative for $\psi(\th)$; such an antiderivative exists because, by \eqref{eq:psi} and conditions \eqref{ell''}, \eqref{ell'''}, together with the conditions on $q$ in the statement of Proposition~\ref{prop:exp}, the function $\psi$ is real-valued and continuous. 

\rule{0pt}{0pt}\big(The continuity of $\psi$ may seem innocuous. However, this condition is very important and takes some effort to obtain, as the reasoning from \eqref{eq:q',q''} to \eqref{eq:psi} shows; one may also note that conditions \eqref{ell'' alt} and \eqref{ell''' alt} are needed only for that reasoning. It would be much easier to derive identity \eqref{eq:ident} from \eqref{eq:main ident} without having to show that an antiderivative for $\psi(\th)$ exists.\big)  

Then equation \eqref{eq:ident} can be rewritten as $\big(s(\th)\ell''_{x\th}(x,\th)\big)'_\th
=0$, so that $T_1(x):=s(\th)\ell''_{x\th}(x,\th)$ does not depend on $\th$. So, letting $w_1(\th):=1/s(\th)=\exp(\frac12\,\int\psi(\th)\dd\th)$, we have $\ell''_{x\th}(x,\th)=w_1(\th)T_1(x)$. Integrating this differential equation in $x$ and $\th$, we conclude that \eqref{eq:exp-fam} holds 
for some functions $c$ and $h$, with $w$ and $T$ being antiderivatives for $w_1$ and $T_1$, respectively. 
The proof of Proposition~\ref{prop:exp} is now complete.  
\end{proof}

\bigskip

%
%
%

By \cite[Theorem~2]{ferguson62}, a family of (say continuous)
probability densities $(p_\th)_{\th\in\R}$ (with respect to the Lebesgue measure on $\XXX=\R$) is simultaneously a location family and a one-parameter exponential family if and only if $p_\th(x)=p(x-\th)$ for all real $x$ and $\th$, where either 
\begin{equation}
	p(u)=q_{\al,\ga}(u):=\frac{|\ga|(\al/e)^\al}{\Ga(\al)}\,\exp\{-\al(e^{\ga u}-1-\ga u)\}
\end{equation}
for some $\al\in(0,\infty)$, some $\ga\in\R\setminus\{0\}$, and all real $u$; or 
\begin{equation}
	p(u)=q_{\si^2}(u):=\frac1{\si\sqrt{2\pi}}\,\exp\Big\{-\frac{u^2}{2\si^2}\Big\}
\end{equation}
for some $\si\in(0,\infty)$ and all real $u$. 

Clearly, $q_{\si^2}$ is the density of $N(0,\si^2)$. One the other hand, $q_{\al,\ga}$ is the density of the r.v.\ $\frac1\ga\,\ln Y_\al$, where $Y_\al$ is a r.v.\ having the Gamma distribution with the shape and scale parameters equal $\al$ and $1/\al$, respectively. 
For any fixed $\si\in(0,\infty)$, letting $\al$ and $\ga$ vary 
so that $\al\to\infty$ and $\al\ga^2\to1/\si^2$, we see 
that, by the central limit theorem, the distribution of $\sqrt\al(Y_\al-1)$ converges to $N(0,1)$ and hence 
the distribution of $\sqrt\al\,\ln Y_\al=\sqrt\al(Y_\al-1)\,\frac{\ln Y_\al}{Y_\al-1}$ converges to $N(0,1)$ as well; so, the distribution of the mentioned r.v.\ $\frac1\ga\,\ln Y_\al$ converges to $N(0,\si^2)$. 
Moreover, using Stirling's formula, one can see (as was done in \cite{ferguson62}) that $q_{\al,\ga}(u)\to q_{\si^2}(u)$ for each real $u$ as $\al\to\infty$ and $\al\ga^2\to1/\si^2$; thus, the normal density $q_{\si^2}$ is a limit case of the ``exponential-Gamma'' density $q_{\al,\ga}$.  

Therefore and in view of Proposition~\ref{prop:exp} (and also Remarks~\ref{rem:fisher} and \ref{rem:compact} and Proposition~\ref{prop:shift}), taking almost any smooth enough location family, except for the normal and ``exponential-Gamma'' ones, one has an example where Theorem~\ref{th:main} of the present paper is applicable, whereas \cite[Theorem~3.16]{nonlinear-publ} is not. 
For instance, one may take the Cauchy location family, with $p_\th(x)=\frac1\pi\,\frac1{1+(x-\th)^2}$, or 
the location family defined by the formula $p_\th(x)=\frac1{2\Ga(5/4)}\,\exp\{-(x-\th)^4\}$, for all real $x$ and $\th$.  
An additional advantage of Theorem~\ref{th:main} of the present paper over \cite[Theorem~3.16]{nonlinear-publ} is that now one does not have to check a special, restrictive condition of the form 
\eqref{eq:q(th)} 
even when it holds.



\bibliographystyle{imsart-number}
\bibliography{C:/Users/ipinelis/Google\string~1/mtu/bib_files/citations12.13.12}


\end{document}